\documentclass[12pt]{amsart}
\usepackage{ amsmath, amsthm, amsfonts, amssymb, color}
 \usepackage{mathrsfs}
\usepackage{amsfonts, amsmath}
 \usepackage{amsmath,amstext,amsthm,amssymb,amsxtra}
 \usepackage{txfonts} %also pxfonts
 \usepackage[colorlinks, citecolor=blue,pagebackref,hypertexnames=false]{hyperref}
 \allowdisplaybreaks
 \usepackage{pgf,tikz}
\begin{document}
\baselineskip 16pt

\newcommand\C{{\mathbb C}}
\newcommand\CC{{\mathbb C^n}}
\newcommand\NN{\mathbb{N}}
\newcommand\ZZ{\mathbb{Z}}
\newcommand\R{{\mathbb R}}
\newcommand\RR{{\mathbb R^n}}

\renewcommand\Re{\operatorname{Re}}
\renewcommand\Im{\operatorname{Im}}

\newcommand{\mc}{\mathcal}
\newcommand\D{\mathcal{D}}

\newtheorem{thm}{Theorem}[section]
\newtheorem{prop}[thm]{Proposition}
\newtheorem{cor}[thm]{Corollary}
\newtheorem{lem}[thm]{Lemma}
\newtheorem{lemma}[thm]{Lemma}
\newtheorem{exams}[thm]{Examples}
\theoremstyle{definition}
\newtheorem{defn}[thm]{Definition}
\newtheorem{rem}[thm]{Remark}

\numberwithin{equation}{section}
\newcommand\bchi{{\chi}}
\newcommand\relphantom[1]{\mathrel{\phantom{#1}}}
\newcommand\ve{\varepsilon}  \newcommand\tve{t_{\varepsilon}}
\newcommand\vf{\varphi}      \newcommand\yvf{y_{\varphi}}
\newcommand\bfE{\mathbf{E}}

\title[Singular Integral Operators on the Fock space]
{A Boundedness Criterion for Singular Integral Operators\\[2pt] of Convolution Type on the Fock Space}

\author[G. F. Cao,   J. Li, M. X Shen, B.D. Wick and L.X. Yan]{Guangfu Cao, \ Ji Li, \
 Minxing Shen, \ Brett D. Wick \ and \ Lixin Yan}
\address{Guangfu Cao, Department of Mathematics, South China Agricultural University, Guangzhou, Guangdong 510640,   P.R. China}
\email{guangfucao@163.com}
\address{Ji Li, Department of Mathematics, Macquarie University, NSW, 2109, Australia}
\email{ji.li@mq.edu.au}
\address{Minxing Shen, Department of Mathematics, Sun Yat-sen (Zhongshan)
University, Guangzhou, 510275, P.R. China}
\email{shenmx3@163.com}
\address{Brett D. Wick, Department of Mathematics, Washington University in St. Louis, St. Louis, MO 63130-4899 USA}
\email{wick@math.wustl.edu}

\address{
Lixin Yan, Department of Mathematics, Sun Yat-sen (Zhongshan) University, Guangzhou, 510275, P.R. China
and Department of Mathematics, Macquarie University, NSW 2109, Australia}
\email{mcsylx@mail.sysu.edu.cn
}

\date{\today}
 \subjclass[2010]{30H20, 42A38, 44A15}
\keywords{ Fock space, Singular integral operator, Bargmann transform, Fourier transform, Riesz transform, Spectrum, Reducing subspace}

\begin{abstract}
We show that  for an entire  function $\varphi$ belonging to the Fock space  ${\mathscr F}^2(\mathbb{C}^n)$ on the complex Euclidean space $\mathbb{C}^n$, the integral operator
\begin{eqnarray*}
S_{\varphi}F(z)=\int_{\mathbb{C}^n} F(w) e^{z \cdot\bar{w}} \varphi(z- \bar{w})\,d\lambda(w),  \ \ \  \ \ z\in \mathbb{C}^n,
\end{eqnarray*}
is bounded on   ${\mathscr F}^2(\mathbb{C}^n)$ if and only if there
exists a function $m\in L^{\infty}(\mathbb{R}^n)$ such that
$$
	\varphi(z)=\int_{\mathbb{R}^n} m(x)e^{-2\left(x-\frac{i}{2}   z  \right)^2} dx, \ \ \ \  \ \ z\in \mathbb{C}^n.
$$
Here $d\lambda(w)=  \pi^{-n}e^{-\left\vert w\right\vert^2}dw$ is the Gaussian measure on $\mathbb C^n$.
With this characterization we are able to obtain some fundamental results of the operator $S_\varphi$, including the normality, the $C^*$ algebraic properties, the spectrum and its compactness. Moreover, we obtain the reducing  subspaces of $S_{\varphi}$.

In particular, in the case $n=1$, we give a complete solution to an open problem proposed by K. Zhu for the Fock space ${\mathscr F}^2(\mathbb{C})$
on the complex plane ${\mathbb C}$ (Integr. Equ. Oper. Theory  {\bf  81} (2015),   451--454).
 \end{abstract}

\maketitle
   \tableofcontents

\section{Introduction }
\setcounter{equation}{0}

The Fock space ${\mathscr F}^2(\CC)$ consists of all   entire  functions $F$ on  the complex Euclidean space $\CC$ such that
$$
\|F\|_{{\mathscr F}^2(\CC)}=\left(\int_{\CC} |F(z)|^2d\lambda(z)\right)^{1\over 2}<\infty,
$$
where
$$
d\lambda(z) = \pi^{-n} e^{-|z|^2} dz
$$
  is the Gaussian measure on $\CC$.
The Fock space ${\mathscr F}^2(\CC)$ is a Hilbert space, whose inner product  is
inherited from $L^2(\CC, d\lambda)$.
This space is a convenient setting for many problems in functional analysis, mathematical physics, and engineering.
We refer to \cite{B1, B2, BC, F, G, Z2, Z1} for an introduction to the theory of Fock spaces and some connections with other areas of mathematics and engineering.

\smallskip

For $\varphi\in \mathscr F^2(\CC)$, consider the integral operator
   \begin{eqnarray} \label{e1.1}
S_{\varphi}F(z)= \int_{\mathbb{C}^n} F(w) e^{z \cdot\bar{w}} \varphi(z- \bar{w}) d\lambda(w).
\end{eqnarray}

\medskip

In 2015, K.  Zhu proposed  the following problem for the Fock space  ${\mathscr F}^2(\C)$ on the complex plane ${\mathbb C}$  (see \cite{Z1}):
    Characterize those functions $\varphi\in  {\mathscr F}^2(\C)$   such that the integral operator
$S_{\varphi}$ in \eqref{e1.1}
is bounded on ${\mathscr F}^2(\C)$.

Two natural conjectures arise from Zhu's question and are related to the ``reproducing kernel thesis'', which roughly says that the behavior of $S_\varphi$ is determined by its action on the normalized reproducing kernels $k_z$ of the Fock space.  Two possible versions of this reproducing kernel thesis one might hope to be true are:
$S_{\varphi}:{\mathscr F}^2(\C)\to {\mathscr F}^2(\C)$ if and only if one of the following conditions holds:
\begin{eqnarray*}
\sup_{z\in\C}\left\Vert S_\varphi k_z\right\Vert_{{\mathscr F}^2(\C)} & < & \infty,\\
\sup_{z\in\C}\left\vert \left\langle S_\varphi k_z,k_z\right\rangle_{{\mathscr F}^2(\C)}\right\vert=\sup_{z\in\C}\left\vert \varphi(z-\overline{z})\right\vert & < & \infty.
\end{eqnarray*}
This strategy is a common, and successful, one to try when working on operator theoretic questions in complex analysis, see \cite{A, BCFEMT, B, MW, PTW, Sm, T}.  While natural, this is unfortunately not true since it is possible to provide a counterexample (provided in Remark \ref{rem3.5} below) to the reproducing kernel thesis in this context, meaning that the exact answer to Zhu's question is more subtle.  In this article, we obtain a complete solution to this open problem using harmonic analysis methods and are further able to resolve the question for the Fock space in all dimensions.  

In \cite{Z1}, via an example, Zhu suggests that there should be some connection between resolving his question and harmonic analysis since he demonstrates that the Hilbert transform is unitarily equivalent to $S_\varphi$ for a special choice of $\varphi$.  From this one example we were lead to guess that Fourier multiplier operators, which are in correspondence with bounded functions, should in fact provide the answer to Zhu's question.
 Indeed, we have the following result on the Fock space ${\mathscr F}^2(\CC)$.

\begin{thm}\label{th1.1}
 The  integral operator $S_\varphi$ in \eqref{e1.1}
  is bounded on ${\mathscr F}^2(\CC)$ if and only if there exists an $m\in L^\infty(\RR)$
such that
\begin{eqnarray}\label{e1.2}
	\varphi(z)=\left({2\over \pi}\right)^{n\over 2}\int_{\RR} m(x) e^{-2\left(x-\frac{i}{2}   z  \right)^2}dx, \ \ \ \  z\in \CC.
\end{eqnarray}
Moreover, we have that
$$ \|S_\varphi\|_{{\mathscr F}^2(\CC)\to {\mathscr F}^2(\CC)} = \|m\|_{L^\infty(\mathbb R^n)}. $$
\end{thm}

The idea of the proof is to utilize the Bargmann transform to reformulate the question as one about a certain operator on $L^2(\RR)$ that is translation invariant.  Then for the operator we have in this context, it will fall into a category of operators well-studied in the harmonic analysis literature, the Fourier multiplier operators, to which we apply the Bargmann transform again and provide the answer to Zhu's question.

\medskip

%We would like to mention that from   \cite[Proposition 2]{Z1},
 %a necessary condition for $S_{\varphi}$ to be
%bounded on ${\mathscr F}^2(\CC)$ is that $\varphi(z- {\bar z})$ is bounded. In other words, the  boundedness of
%$S_{\varphi}$ implies that the function $\varphi$ is bounded on the imaginary axis.

%As a direct corollary, we can characterize the accurate norm of  $S_{\varphi}$, which is not known before.
%\begin{cor}\label{cor1.2}
%For $\varphi(z)=\int_{\RR} m(x)e^{-2\left(x-\frac{i}{2}   z  \right)^2} dx$ with some $m\in L^\infty(\mathbb R^n)$, we have that
%$$ \|S_\varphi\| = \|m\|_{L^\infty(\mathbb R^n)}. $$
%\end{cor}

%As applications, we also establish some important properties of $S_{\varphi}$.
%\begin{thm}\label{th1.3}
% For any $m\in L^\infty(\mathbb R^n)$,  $S_\varphi$ is a normal operator on $\mathscr F(\mathbb C^n)$ with $\varphi$ defined as in \eqref{e1.2}.
% \end{thm}

With the characterization in Theorem \ref{th1.1} we are able to obtain some fundamental operator theory results about $S_{\varphi}$.  In particular, we are able to determine the normality of $S_{\varphi}$, the spectrum of an individual $S_{\varphi}$ and the reducing  subspaces of $S_{\varphi}$.  A particular corollary of our work is:
\begin{thm}\label{th1.3}
 Suppose $\varphi \in\mathscr F^2(\mathbb C^n)$ such that $S_\varphi$ is bounded on $\mathscr F^2(\mathbb C^n)$, then $S^*_\varphi=S_{\tilde\varphi}$, where $\varphi$ is as in \eqref{e1.2} and
 $$ \tilde\varphi (z) =\left({2\over \pi}\right)^{n\over 2} \int_{\mathbb R^n} \overline { m(x) }
  e^{-2\left(x-\frac{i}{2}   z  \right)^2} dx. $$
 Furthermore, $S_\varphi$ is normal.
 \end{thm}

In the last decades, Toeplitz operators, Hankel operators and composition operators on several analytic
function spaces (Hardy spaces, Bergman spaces, Dirichlet spaces and Fock spaces) have been widely studied. For example,
one may consult  the references \cite{BC,BH, Do, Sa}. It  is well-known that these operators are never normal if
their symbols are analytic. For example, if $\varphi $ is a bounded analytic function on the unit disc $\mathbb{D}$ in
the complex plane $\mathbb{C}$, or unit ball $\mathbb{B}_n$ in the complex space $\mathbb{C}^n$, then $T_{\varphi}$,
the Toeplitz operator on the Hardy space $H^2(\mathbb{D})$ or $H^2(\mathbb{B}_n)$ , is normal if and only if $\varphi $
is a constant.  However, $S_{\varphi }$ is always normal although $\varphi$ is
analytic, this is a surprising phenomenon.  For the other operator theory results that are immediate corollaries of
Theorem \ref{th1.1} and Theorem \ref{th1.3} we refer to Section \ref{s:Operator}.

We provide two remarks regarding our main results  Theorem \ref{th1.1} and \ref{th1.3}, on the extension to the Fock space $\mathscr F^2_{\alpha}(\CC)$ and on the boundedness on the Fock space
${\mathscr F}^p(\CC)$  for $p\in [1,\infty)$, respectively.
\begin{rem}\label{rem1.1}
There are natural extensions of the results in Theorem \ref{th1.1} and Theorem \ref{th1.3} to the Fock space $\mathscr F^2_{\alpha}(\CC)$, where
$$
\|F\|_{{\mathscr F}_{\alpha}^2(\CC)}=\left(\int_{\CC} |F(z)|^2d\lambda_{\alpha}(z)\right)^{1/2}<\infty,
$$
and
$$
d\lambda_{\alpha}(z) = \pi^{-n} e^{-\alpha|z|^2} dz
$$
with $\alpha>0$.  We don't precisely formulate these results since the modifications necessary to do so are standard.
\end{rem}

\begin{rem}\label{rem1.2}
It is natural to ask whether the characterization of $S_\varphi$ as in Theorem \ref{th1.1} can imply boundedness of $S_\varphi$ on the Fock space
${\mathscr F}^p(\CC)$ for $p\in [1,\infty)$, where ${\mathscr F}^p(\CC)$ consists of all entire  functions $F$ on  the complex Euclidean space $\CC$ such that
$$
\|F\|_{{\mathscr F}^p(\CC)}=\left(\int_{\CC} |F(z)|^pd\lambda(z)\right)^{1\over p}<\infty.
$$
However, this is not true for $p\in[1,2)$. We will provide a counterexample in Section \ref{s:Proof}.  The reader should not confuse the definition we use here with the other definition in the literature of those entire functions such that $f(z)e^{-\frac{\left\vert z\right\vert^2}{2}}$ belong to $L^p(\CC)$. 
{See for example \cite{CHH}. We point out that the operatoter $S_\varphi$ may not be well-defined in the other Fock space for $2<p<\infty$. Explanations will be provided in Section  \ref{s:Proof}.} 
%{\color{red}See also \cite{CHH} for another definition of the Fock space ${\mathscr F}^p(\CC)$.}
\end{rem}

The outline of the remainder of the paper is as follows.  In Section \ref{s:prelim} we collect the basic definitions and concepts
that we will need to prove the main result.  In Section \ref{s:Proof} we give the proof of the main result and in Section \ref{s:Apps}
we show how the main result can recover the known examples in the literature and can further recover some canonical Calder\'on--Zygmund operators.
In Section
\ref{s:Operator}
we study operator theoretic properties of the operator $S_\varphi$, including the normality,
$C^*$ algebraic properties, the compactness, the spectrum and the reducing subspaces. In the final section we provide some concluding remarks.

%\begin{cor}\label{cor1.4}
%For any $m\in L^\infty(\mathbb R^n)$, we have $\sigma(S_\varphi) =\mathcal R (m)(\mathbb R^n) $, where $\mathcal R (m)(\mathbb R^n)$ is the essential range of $m$.
%\end{cor}

\section{Preliminaries}
\setcounter{equation}{0}
\label{s:prelim}

{We now set the notation and some common concepts to be used throughout the course of the paper.}  $\RR $ denotes the real Euclidean space and $\CC$ denotes the complex Euclidean space. To simplify the dot product notation, we will denote by simple juxtaposition:  $xy=x \cdot y=\sum_{j=1}^n x_jy_j.$ In particular, this implies that $x^2=x \cdot x=\sum_{j=1}^n x_j^2$.  The Hermitian inner product in $\CC$ will be denoted by
$z{\bar w}$ when $z,w\in\CC$; this then gives $|z|^2=z{\bar z}= \sum_{j=1}^n |z_j|^2$.  The standard norm on the Lebesgue space $L^2(\RR)$ will be denoted by $\|f\|_2=\|f\|_{L^2(\RR, dx)}$.  And, as introduced earlier, the Fock space on $\CC$ will be denoted by $\mathscr F^2(\CC)$ with the norm:
$$
\|f\|_{{\mathscr F}^2(\CC)}=\left(\int_{\CC} |F(z)|^2d\lambda(z)\right)^{1/2}
$$
where $d\lambda(z) = \pi^{-n} e^{-|z|^2} dz$.

A fundamental tool in our analysis is the Fourier transform of a function $f$, i.e.
$$
{\mathcal F} f(x) =\pi^{-{n\over2}} \int_{\RR} e^{-2ix\cdot y}f(y) dy, \ \ \ \ x\in\RR.
$$
The inverse of the Fourier  transform ${\mathcal F}$ will be denoted by ${\mathcal F}^{-1}$, i.e,
 ${\mathcal F}{\mathcal F}^{-1}={\mathcal F}^{-1}{\mathcal F}=Id$, the identity operator on $L^2(\RR)$.

\subsection{The Fock Space}

We start by recalling some basic facts about the Fock space.
Throughout the paper,
we denote the scalar product on ${\mathscr F}^2(\mathbb C^n)$ by
$\langle \cdot , \cdot \rangle_{{\mathscr F}^2(\CC)}$.
It is well-known (see for example, \cite[Theorem 1.63]{F}) that  the collection of monomials of the form
$$
e_{\alpha} (z)=\left({1\over \alpha!}\right)^{1\over2} z^{\alpha}= \prod_{j=1}^n \left({1\over \alpha_j!}\right)^{1\over2} z_j^{\alpha_j}
$$
for all $\alpha=(\alpha_1, \ldots, \alpha_n)$ with $\alpha_j\geq 0$, forms  an orthonormal basis for ${\mathscr F}^2(\mathbb C^n)$.
This space $\mathscr F^2(\mathbb C^n)$  is a reproducing kernel Hilbert space, that is
$$
|F(z)|\leq e^{ |z|^2\over2} \|F\|_{{\mathscr F}^2(\CC)},  \ \ \ {\rm for\ all}\ z\in \CC.
$$
The reproducing kernel of $\mathscr F^2(\mathbb C^n)$ is
\begin{eqnarray}\label{mm}
 K(z,{\bar w}) = \sum_\alpha e_\alpha(z) \overline{ e_\alpha(w)}  = \sum_\alpha { z^\alpha\cdot \bar w^\alpha \over \alpha! }  = e^{z\cdot \bar w},
  \end{eqnarray}
so that  $\|K(z,\cdot)\|^2_{\mathscr F^2(\mathbb C ^n)} = e^{|z|^2}$  and
\begin{eqnarray}\label{mmm}
 F(z) = \int_{\mathbb C^n} F(w) e^{z\cdot \bar w} d\lambda(w), \ \ \  z\in\CC
 \end{eqnarray}
when $F\in \mathscr F^2(\mathbb C^n)$.

An important consequence of the existence of a reproducing kernel is that every bounded operator
$T$ on $\mathscr F^2(\mathbb C^n)$ can be written as an integral operator. More precisely we have
\begin{prop}[\cite{F}]\label{prop Folland}
If $T$ is a bounded operator on $\mathscr F^2(\mathbb C^n)$, let $K_T(z,\bar w) = T K(\cdot,\bar w)(z)$.
Then $K_T$ is an entire function on $\mathbb C^{2n}$ that satisfies
\begin{itemize}
\item[(a)] $K_T(\cdot, w) \in\mathscr F^2(\mathbb C^n)$ for all $w$ and $K_T(z,\cdot) \in\mathscr F^2(\mathbb C^n)$ for all $z$;

\item[(b)] $|K_T(z, \bar w)|\leq e^{|z|^2+|w|^2} \|T\|$;

\item[(c)] $TF(z) = \int_{\mathbb C^n} K_T(z, \bar w) F(w) d\lambda(w)$ for all $F\in \mathscr F^2(\mathbb C^n)$ and $z\in\mathbb C^n$.
\end{itemize}
\end{prop}

As we can see from this proposition, the form of the kernel in \eqref{e1.1} is 
\begin{align}\label{kernel K_T}
 K_T(z, \bar w) = e^{z\cdot \bar w} \varphi(z-\bar w).
\end{align}
%In Theorem \ref{th1.1} we provide a characterization of $\varphi$ such that the operator $T=S_\varphi$ is bounded on $\mathscr F^2(\mathbb C^n)$.

\medskip

\subsection{The Bargmann Transform}
The  Bargmann transform is an old tool in mathematics analysis and mathematical physics
(see \cite{B1, B2, F, G, Se, Z1, Z2, Z3} and references therein).
   Consider $f\in L^2(\RR)$, and define
\begin{eqnarray}\label{e2.1}
 Bf(z)&=& \left({2\over \pi}\right)^{n\over 4}\int_{\RR} f(x) e^{2x\cdot z-x^2- {z^2\over 2}} dx\nonumber \\
 &=& \left({2\over \pi}\right)^{n\over 4} e^{ {z^2\over 2}}  \int_{\RR} f(x) e^{-(x-z)^2}dx, \ \ \ \ z\in \CC.
\end{eqnarray}
  Since the function $e^{2x\cdot z-x^2-(z^2/2)}$ is in $L^2(\RR)$, the integral is absolutely convergent in
$L^2(\RR)$. Using Morera's theorem one may verify that $Bf$ is an entire holomorphic function on $\CC$.
From \eqref{e2.1} one sees that the Bargmann transform is very closely related to the Fourier transform or the Fourier-Wiener
transform (see \cite{F, G}). %In considering $Bf(x+iy)$, one should interpret $x$ as a position variable and $y$ as a frequency variable. In the context of quantum mechanics, the frequency variable has the interpretation of momentum (see \cite{F, G}).
%See \cite{B1, B2, F, G, Se, Z1, Z2} for more information on the Bargmann transform.
%By a standard argument (see for example, \cite{G}), we have the following result.

The following result  is well-known (see for example, \cite{G}).
\begin{lemma} \label{le2.3}
The Bargmann transform is a unitary operator from $L^2(\RR)$ onto ${\mathscr F}^2(\CC)$: it is one-to-one, onto, and isometric
in the sense that
\begin{eqnarray*}\label{e2.2}
\int_{\RR} |f(x)|^2dx =\int_{\CC} |Bf(z)|^2 d\lambda(z).
\end{eqnarray*}
 \end{lemma}

\begin{proof}
For the proof, we refer to \cite[Proposition 3.4.3]{G}.
\end{proof}

Let us now compute the inverse Bargmann transform. Since $B$ is unitary, for $F\in {\mathscr F}^2(\CC)$ and
$g\in L^2(\RR)$, by \eqref{e2.1} we have
$$
\langle B^{-1} F, g\rangle_{L^2(\RR)} =\langle F, Bg\rangle_{{\mathscr F}^2(\CC)} =
 \left({2\over \pi}\right)^{n\over 4}   \int_{\mathbb{C}^n} F(z) \int_{\mathbb{R}^n}    {\overline g(x)}
     e^{2x\cdot \bar{z}-x^2-  {{\bar z}^2\over2}  } dx\,  d\lambda(z),
$$
and hence
\begin{eqnarray}\label{e22.1}
 B^{-1} F(x)&=& \left({2\over \pi}\right)^{n\over 4}\int_{\CC} F(z) e^{2x\cdot {\bar z}-x^2- {{\bar z}^2\over 2}} d\lambda(z),\quad x\in\RR. %\nonumber \\
% &=& \left({2\over \pi}\right)^{n\over 4}   \int_{\CC} F(z) e^{-(x-{\bar z})^2}e^{ {{\bar z}^2\over 2}}d\lambda(z), \ \ \ \ x\in \RR.
\end{eqnarray}

%To study the  Bargmann transform of Riesz transform, we first consider the

To prove our main result Theorem \ref{th1.1}, we need to study the Bargmann transform of the Fourier transform (a bounded operator on $L^2(\RR)$) and inverse Fourier transform (also a bounded operator on $L^2(\RR)$).

\begin{lem}\label{prop2.2}
For every $F\in {\mathscr F}^2(\CC)$ and $z\in \CC $, we have
\begin{align*}
B \mathcal{F} B^{-1}F(z)= F(-iz), \quad{\rm and}\quad B \mathcal{F}^{-1} B^{-1}F(z)= F(iz).
 \end{align*}
\end{lem}

\begin{proof}
This lemma  was proved  in \cite[Theorem 3]{DZ} for  the case $n=1$. See also  \cite[Theorem 4]{Z3}.  We give a brief proof of this lemma in the higher dimensional case for completeness and the convenience of the reader.

%Recall that for $F \in {\mathscr F}^2(\CC)$, we have
%$$  B^{-1}F(t)=c\int_{\CC} F(w) e^{2t \cdot \bar w-t^2-({\bar w}^2/2)} d\lambda(w), \quad t\in \mathbb R^n.
%\ $$

By taking the Fourier transform, we have
\begin{align*}
\mathcal{F} B^{-1}F(\xi)&=\pi^{-{n\over2}} \int_{\RR} e^{-2i\xi\cdot t}B^{-1}F(t) dt
%& = \pi^{-n/2} \int_{\RR} e^{-2i\xi\cdot t} \int_{\CC} F(w) e^{2t \cdot \bar w-t^2-({\bar w}^2/2)} d\lambda(w)\  dt\\
%& = \left({2\over \pi}\right)^{n\over 4} \pi^{-n/2}  \int_{\CC} F(w)e^{-({\bar w}^2/2)}    \int_{\RR} e^{-2i\xi\cdot t} e^{2t \cdot \bar w-t^2}   dt\ d\lambda(w)\\
 = 2^{n\over4}   \pi^{-{3n\over4}}  \int_{\CC} F(w)e^{-{{\bar w}^2\over2}}   e^{(\bar w - i\xi)^2}  \int_{\RR} e^{- (t- (\bar w - i\xi))^2}   dt\ d\lambda(w).
\end{align*}
Recall that by a change of variables and standard calculus computations,
$$\int_{\RR} e^{- (t- (\bar w - i\xi))^2}   dt = \pi^{n\over2}.$$
We then have
\begin{align}\label{F-1BF}
\mathcal{F} B^{-1}F(\xi)
& = \left({2\over \pi}\right)^{n\over 4} \int_{\CC} F(w)e^{-{{\bar w}^2\over2}}   e^{(\bar w - i\xi)^2}  d\lambda(w)\nonumber\\
 &= \left({2\over \pi}\right)^{n\over 4}  e^{-\xi^2} \int_{\CC} F(w)e^{{{\bar w}^2\over2}}   e^{-2i\bar w \cdot \xi}  d\lambda(w).
\end{align}

Then, by taking the Bargmann transform of $\mathcal{F} B^{-1}F$ we get that
\begin{align*}
B \mathcal{F} B^{-1}F(z)
%&= \left({2\over \pi}\right)^{n\over 4}\int_{\RR} \mathcal{F} B^{-1}F(\xi)\ e^{2\xi\cdot z-\xi^2-(z^2/2)} d\xi \\
%&= \Big({2\over \pi}\Big)^{n\over 2}\int_{\RR}\  e^{-\xi^2} \int_{\CC} F(w)e^{({\bar w}^2/2)}   e^{-2i\bar w \cdot \xi}  d\lambda(w)\ e^{2\xi\cdot z-\xi^2-(z^2/2)} d\xi \\
%& =\Big({2\over \pi}\Big)^{n\over 2} e^{-(z^2/2)}\int_{\CC} F(w)e^{({\bar w}^2/2)}  \int_{\RR}\      e^{2\xi\cdot (z- i\bar w )-2\xi^2} d\xi \ d\lambda(w)\\\
& =\left({2\over \pi}\right)^{n/ 2} e^{-{z^2\over2}}\int_{\CC} F(w)e^{{{\bar w}^2\over2}}  e^{ (z-i\bar w)^2\over 2  } \int_{\RR}\      e^{-2(\xi-{z-i\bar w\over2})^2} d\xi \ d\lambda(w)\\
% \end{align*}
%By noting that
%$$\int_{\RR}\      e^{-2(\xi-{z-i\bar w\over2})^2} d\xi =\Big({\pi\over2}\Big)^{n\over2}, $$
%we get that
%\begin{align*}
%B \mathcal{F} B^{-1}F(z)
& =%\Big({\pi\over2}\Big)^{n\over2} \Big({2\over \pi}\Big)^{n\over 2}
e^{-{z^2\over2}}\int_{\CC} F(w)e^{{{\bar w}^2\over2}}  e^{ (z-i\bar w)^2\over 2  } \ d\lambda(w)\\
&= \int_{\CC} F(w) e^{ (-iz)\cdot\bar w  } \ d\lambda(w)\\
&= F(-iz),
 \end{align*}
where the last equality follows from the reproducing formula.

By repeating the above proof, we also have
\begin{align*}
B \mathcal{F}^{-1} B^{-1}F(z)= F(iz).
 \end{align*}
 The proof of Lemma \ref{prop2.2} is complete.
\end{proof}

%\bigskip
%
%\bigskip
%\bigskip
%\bigskip
%\bigskip
%\bigskip
%
%\bigskip
%\bigskip
%\bigskip
%\bigskip
%
%
%
%\bigskip
%\bigskip
%
%\bigskip
%\bigskip
%\bigskip
%\bigskip

%\medskip

\section{Proof of Theorem~\ref{th1.1}}
\setcounter{equation}{0}
\label{s:Proof}

In this section we provide the proof of our main result Theorem \ref{th1.1}.
To begin with, we need the following auxiliary result.
\begin{lemma}\label{le2.1}
For any $m\in L^\infty(\RR)$, the entire  function
$$
\varphi(z)=\int_{\RR} m(x)e^{-2(x-\frac{i}{2}z)^2} dx, \ \ \ \ \  z\in\CC
$$
belongs to ${\mathscr F}^2(\CC)$.
\end{lemma}

\begin{proof}
For every $z\in\CC$, we write  $z=u+iv$. Then we have
\begin{eqnarray*}
\varphi(z)
&=&\int_{\RR} m(x-\frac 12 v)e^{-2x^2+2iu\cdot x+\frac 12 u^2} dx 
= \pi^{n\over2} \mathcal{F}^{-1} \big[m(x-\frac 12 v)e^{-2x^2}\big](u) e^{\frac 12 u^2}.
\end{eqnarray*}
By Plancherel's  theorem,
\begin{eqnarray*}
||\varphi||_{{\mathscr F}^2(\CC)}^2
&=&\pi^{-n}  \int_{\CC} |\varphi(z)|^2 e^{-|z|^2}dz\\
&=& \int_{\RR} e^{-v^2} dv \int_{\RR} \Big|\mathcal{F}^{-1}\big[m(x-\frac 12 v)e^{-2x^2}\big](u)\Big|^2 du\\
&=&  \int_{\RR} e^{-v^2} dv \int_{\RR} \big|m(x-\frac 12 v)e^{-2x^2}\big|^2 dx\\
&\leq&  ||m||_{L^\infty}^2 \int_{\RR} e^{-v^2} dv \int_{\RR} e^{-4x^2} dx <\infty,
\end{eqnarray*}
and so $\varphi \in {\mathscr F}^2(\CC).$ This finishes the proof of Lemma~\ref{le2.1}.
\end{proof}

The proof of Theorem~\ref{th1.1} relies on the following elementary fact taken from harmonic analysis characterising the translation invariant operators that are bounded on $L^2(\RR)$.

\begin{prop}\label{prop3.1}
Let $T$ be a bounded linear transformation mapping $L^2(\RR)$ into itself. Then a necessary and sufficient condition that
$T$ commutes with translation  is that there exists a bounded measurable function $m(y)$ ( a ``multiplier") so that
$ {\mathcal F ({Tf}})(y)= m(y)  {\mathcal F {f}}(y)$ for all $f\in L^2(\RR)$.  In this case
the norm of $T : L^2(\RR)\to L^2(\RR)$  is equal to $\|m\|_{L^{\infty}}$.
\end{prop}

\begin{proof}
For the proof of this proposition see \cite[Proposition 2, Chapter 2]{St}.
\end{proof}

For more information on the translation invariant operators, we refer to \cite{Ho} and \cite[Chapter 1]{SW}.  In the following  we  denote by ${\mathscr M}^{2,2}(\RR)$ the set of all bounded linear operators on $L^2(\RR)$  that commute with translations.

\medskip

Recall that the operators $B$ and $B^{-1}$ are the Bargmann transform in \eqref{e2.1} and the inverse Bargmann transform in \eqref{e22.1}, respectively.
For every bounded operator  $S_\varphi$ in \eqref{e1.1}  on the space ${\mathscr F}^2(\CC)$, consider the operator
\begin{eqnarray}\label{tt}
 T= B^{-1} S_\varphi B.
\end{eqnarray}
A crucial observation is that the above operator $ T$   commutes with translation so that we can apply Proposition~\ref{prop3.1} in the proof of Theorem~\ref{th1.1}.  To be precise, we first  have the following result.

\begin{lemma}\label{le3.2}
If the integral operator $S_\varphi$ in \eqref{e1.1}  is bounded on ${\mathscr F}^2(\CC)$, then there exists an operator  $T\in\mathscr {M}^{2,2}(\RR)$ such that
\begin{eqnarray}\label{e3.1}
	S_{\varphi}F(z)=BTB^{-1}F(z),
\end{eqnarray}
for $F\in {\mathscr F}^2(\CC)$ and $z\in \CC.$  Moreover, there exists a bounded measurable function $m(y)$   so that
${\mathcal F ({Tf}})(y)= m(y)  {\mathcal F {f}}(y)$ for all $f\in L^2(\RR)$.
\end{lemma}

\begin{proof} Let  $T$  be the operator given  in \eqref{tt}.  Then the operator   $T$
  is bounded on $L^2(\RR) $
since  the  Bargmann transform $B$ is unitary operator from $L^2(\RR)$ to ${\mathscr F}^2(\CC)$
%on $\CC$,
and $S_\varphi$ in \eqref{e1.1}  is bounded on ${\mathscr F}^2(\CC)$.

Let us show that $T$ commutes with translation.  To do so, let $\tau_a$ denote the translation operator by $a\in\mathbb R^n$ that acts on a function $f$ as
$$
\left(\tau_af\right)(x)=f(x-a). %, \ \ \ \ \  a\in \RR.
$$
By the definition of the operators $B$ and $B^{-1}$,
$$
B\tau_{a}B^{-1} (F)(z)=F(z-a)e^{z\cdot a-{a^2\over2}}=:W_aF(z).
$$
Then we have
\begin{eqnarray}\label{e3.2}
 \tau_aT&=&B^{-1} \left(B\tau_{a}B^{-1} \right)   S_\varphi B = B^{-1}W_aS_{\varphi}B
 \end{eqnarray}
 and
\begin{eqnarray}\label{e3.3}
 T\tau_a&=&B^{-1} S_\varphi \left( B\tau_a B^{-1} \right) B = B^{-1}S_{\varphi}W_aB.
\end{eqnarray}
 A  straightforward calculation shows that
\begin{eqnarray*}
W_aS_{\varphi}F(z)
%&=&W_a\int_{\CC} f(w) e^{z\bar{w}} \varphi(z-\bar{w}) d\lambda(w)\\
&=&\int_{\CC} F(w) e^{(z-a)\cdot \bar{w}} \varphi((z-a)-\bar{w}) e^{z\cdot a-{a^2\over2}} d\lambda(w)\\
&=& \pi^{-n}  e^{-{a^2\over2}} \int_{\CC} F(w)\varphi((z-a)-\bar{w}) e^{(z-a)\cdot \bar{w}+z\cdot a-|w|^2} dw\\
&=&  \pi^{-n}  e^{-{a^2\over2}} \int_{\CC} F(u-a)\varphi(z-\bar{u}) e^{(z-a)\cdot (\bar{u}-a)+z\cdot a-(u-a)\cdot (\bar{u}-a)} du\\
&=&   e^{-{a^2\over2}}  \int_{\CC} F(u-a)\varphi(z-\bar{u}) e^{u\cdot a+z\cdot \bar{u}} d\lambda(u),
\end{eqnarray*}
and
\begin{eqnarray*}
S_{\varphi}W_aF(z)
&=&\int_{\CC} F(w-a) e^{w\cdot a-{a^2\over2}} e^{z\cdot \bar{w}} \varphi(z-\bar{w})\, d\lambda(w)\\
&=&  e^{-{a^2\over2}}\int_{\mathbb{C}} F(w-a)\varphi(z-\bar{w}) e^{w\cdot a+z\cdot \bar{w}} d\lambda(w),
\end{eqnarray*}
and so $W_aS_{\varphi}=S_{\varphi}W_a.$
This, in combination with \eqref{e3.2} and \eqref{e3.3}, shows that $T$ commutes with translation, and so  $T \in {\mathscr M}^{2,2}(\RR)$.  By Proposition~\ref{prop3.1}, there exists a bounded measurable function $m(y)$   so that
$ {\mathcal F ({Tf}})(y)= m(y)  {\mathcal F {f}}(y)$ for all $f\in L^2(\RR)$.
 The proof of Lemma~\ref{le3.2} is complete.
\end{proof}

Further, we have the following result.

\begin{lemma}\label{le3.3}
If $T \in {\mathscr M}^{2,2}(\RR)$ is given by convolution such that
 $ {\mathcal F ({Tf}})(y)= m(y)  {\mathcal F {f}}(y)$ with an $L^{\infty}(\RR)$ function $m$ and for all $f\in L^2(\RR)$, then for every $ F\in{\mathscr F}^2(\CC)$,
\begin{eqnarray}\label{e3.4}
	BTB^{-1}F(z)=  \left({2\over \pi}\right)^{n\over2} \int_{\CC} F(w)e^{z\cdot \bar{w}}
\left(\int_{\RR} m(x)e^{-2(x-\frac{i}{2}(z-\bar{w}))^2} dx\right) d\lambda(w),  \ \ \ z\in\CC.
\end{eqnarray}
%where   $m$ is an $L^{\infty}$ function of the Fourier transform of $u$.
\end{lemma}

\begin{proof}
By Lemma~\ref{prop2.2},
$$
(B\mathcal{F}B^{-1}F)(z)=F(-iz).
$$
This gives
\begin{eqnarray*}
(\mathcal{F}B^{-1}F)(x)
&=&B^{-1}(B\mathcal{F}B^{-1}F)(x)
= \left({2\over \pi}\right)^{n\over 4}\int_{\CC} F(w) e^{-2ix\cdot \bar{w}-x^2+{(\bar{w})^2\over2}} d\lambda(w),
\end{eqnarray*}
and so
\begin{eqnarray*}
B(m\mathcal{F}B^{-1}F)(z)
&=& \left({2\over \pi}\right)^{n\over 4}\int_{\RR}  m(x) \Big( \mathcal{F}B^{-1}F\Big)(x)   e^{2x\cdot z -x^2 -{z^2\over2}} dx \\
&= & \left({2\over \pi}\right)^{n\over 2} \int_{\CC} F(w)e^{-iz\cdot \bar{w}} \int_{\RR} m(x)e^{A(x,z,w)} dxd\lambda(w),
\end{eqnarray*}
where
\begin{eqnarray*}
A(x,z,w)
&= &-2x^2-{z^2\over2}+2x\cdot z+{\bar{w}^2\over2}-2ix\cdot \bar{w}+iz\cdot \bar{w}\\
&= &-2x^2+2x\cdot (z-i\bar{w})-{(z-i\bar{w})^2\over2}\\
&= &-2\left(x-\frac{z-i\bar{w}}{2}\right)^2.
\end{eqnarray*}
By Lemma~\ref{prop2.2} again,
$$
(B\mathcal{F}^{-1}B^{-1}F)(z)=F(iz).
$$
Therefore,
\begin{eqnarray*}
BTB^{-1}F(z)
&= &(B{\mathcal F}^{-1}(m{\mathcal F}B^{-1}F))(z)
= (B{\mathcal F}^{-1}B^{-1})B(m{\mathcal F}B^{-1}F)(z)
= B(m\mathcal{F}B^{-1}F)(iz)\\
&= &\left({2\over \pi}\right)^{n\over 2}  \int_{\CC} F(w)e^{z\cdot \bar{w}} \left(\int_{\RR} m(x)e^{-2(x-\frac{i}{2}(z-\bar{w}))^2} dx\right) d\lambda(w).
\end{eqnarray*}
The proof of Lemma~\ref{le3.3} is complete.
\end{proof}

\medskip
Now we are ready to prove our main result, Theorem~\ref{th1.1}.
\begin{proof}[Proof of Theorem~\ref{th1.1}]
Assume that  the   operator $S_\varphi$ in \eqref{e1.1}  is bounded on ${\mathscr F}^2(\CC)$.  Let us show that
there exists an $m\in L^\infty(\RR)$
such that \eqref{e1.2} holds.
 Indeed, it follows from Lemma~\ref{le3.2} and Lemma~\ref{le3.3} that there exists an $L^{\infty}(\RR)$ function $m$ such that
for every $ z\in \CC$,
\begin{eqnarray}\label{bbb}
S_\varphi(F)(z)&=& BTB^{-1}(F)(z)\nonumber\\
&=&  \left({2\over \pi}\right)^{n\over 2} \int_{\CC} F(w)e^{z\cdot \bar{w}}
\left(\int_{\RR} m(x)e^{-2\big(x-\frac{i}{2}(z-\bar{w})\big)^2} dx\right) d\lambda(w).
\end{eqnarray}
Define
\begin{eqnarray}\label{e3.5}
\varphi_0(z)=\left({2\over \pi}\right)^{n\over2}
 \int_{\RR} m(x)e^{-2(x-\frac{i}{2}z)^2} dx.
\end{eqnarray}
By Lemma~\ref{le2.1}, we have that
$\varphi_0\in {\mathscr F}^2(\C^n)$.

Let  $\varphi$ be an entire function in \eqref{e1.1}.
We now  show that $\varphi=\varphi_0$.
Indeed, we take  $z=0$ in \eqref{e1.1} and \eqref{bbb} to see that for all $F\in {\mathscr F}^2(\CC)$
\begin{eqnarray}\label{zz}
	\int_{\mathbb{C}^n} F(w)\Big(\varphi(-\bar{w})-\varphi_0(-\bar{w})\Big)\, d\lambda(w) =0.
\end{eqnarray}
Notice that  $\psi(w)=\varphi(-w)-\varphi_0(-w) \in \mathscr{F}^2(\mathbb{C}^n)$. From the standard orthonormal basis $\{e_{\alpha}(z)\}_{\alpha} $ for ${\mathscr{F}^2(\mathbb{C}^n)}$,  we decompose $\psi$ into the series
$$
	\psi(w)=\sum\limits_{\alpha} c_{\alpha} e_{\alpha} (w) =  \sum\limits_{\alpha} c_{\alpha}\left(  {1\over \alpha! }\right)^{1\over2}\, w^{\alpha },
$$
with  $\sum_{\alpha} |c_{\alpha}|^2 = \|\psi\|^2_{{\mathscr F}^2(\CC)}$.  Define
$$
	 \Psi(w) = \sum\limits_{\alpha} \bar{c}_{\alpha} \left(  {1\over \alpha! }\right)^{1\over2}\, w^{\alpha },
$$
where  $\bar{c}_{\alpha}$ is  the complex conjugate of ${c}_{\alpha}$, so that $\psi({\bar w} )={\overline {\Psi({  w} )}}$ and $\|\psi\|^2_{{\mathscr F}^2(\CC)}=\|\Psi\|^2_{{\mathscr F}^2(\CC)}$.  By   \eqref{zz},
\begin{eqnarray}\label{zzz}
	0= \int_{\mathbb{C}^n} F(w)\psi({\bar w} ) d\lambda(w) = \int_{\mathbb{C}^n} F(w) {\overline {\Psi({  w} )}} d\lambda(w).
\end{eqnarray}
Letting $F=\Psi$ in \eqref{zzz}, we see  that  ${\Psi({  w} )}=0 $ for all $w\in \CC$, and so $\psi(w)=0$.    Hence,
 $$
 	\varphi(z)=\varphi_0(z)=\left( {2\over\pi}\right)^{n\over2} \int_{\RR} m(x)e^{-2(x-\frac{i}{2}z)^2} dx
 $$
as desired.

%Indeed, we take  $z=0$ in \eqref{e1.1} and \eqref{bbb} to see that for all $F\in {\mathscr F}^2(\CC)$
%$$
%\int_{\CC} F(w) \, \overline{\big(\varphi(-w)- \varphi_0(-w)\big)} d\lambda(w)= 0.
%$$
%Letting $F(w)=\varphi(-w)- \varphi_0(-w)$, we see  that  $\varphi(-w)- \varphi_0(-w)=0 $ for all $w\in \CC$.  Hence,
% $$
% 	\varphi(z)=\varphi_0(z)=\left( {2\over\pi}\right)^{n\over2} \int_{\RR} m(x)e^{-2(x-\frac{i}{2}z)^2} dx
% $$
%as desired.

Next,   assume
  that \eqref{e1.2} holds for  some $m\in L^\infty(\RR)$. Then Lemma \ref{le2.1} shows that
 the function $\varphi$ as in \eqref{e1.2} is an entire function in ${\mathscr F}^2(\CC)$.
 For the   operator $S_\varphi$ in \eqref{e1.1}, we apply Lemma~\ref{le3.3}
 to obtain
  \begin{eqnarray*}
  S_\varphi=BTB^{-1},
  \end{eqnarray*}
  where $T \in {\mathscr M}^{2,2}(\RR)$ is given by convolution
   such that
 $({\mathcal F {Tf}})(y)= m(y)  {\mathcal F {f}}(y)$ for an $L^{\infty}(\RR)$ function $m$ and for all $f\in L^2(\RR)$.  From the properties of the operators $B$ and $B^{-1}$,
  the operator $  S_\varphi$
   is bounded on the space ${\mathscr F}^2(\CC)$.

To conclude, we point out that by using $S_\varphi = B T B^{-1} $, one obtains
\begin{align*}
\|S_\varphi\|_{{\mathscr F}^2(\CC)\to{\mathscr F}^2(\CC)} &= \| B T B^{-1}\|_{{\mathscr F}^2(\CC)\to{\mathscr F}^2(\CC)}
 = \|T\|_{L^2(\mathbb R^n)\to L^2(\mathbb R^n)}= \|m\|_{L^\infty(\mathbb R^n)}.
\end{align*}
 The proof of Theorem~\ref{th1.1} is complete.
\end{proof}

%We now provide the proof of Corollary \ref{cor1.2}.
%\begin{proof}[Proof of Corollary \ref{cor1.2}]
%\end{proof}
%
%\begin{proof}[Proof of Theorem \ref{th1.3}]
%Assume $T\in \mathcal M^{2,2}$ with $\widehat {Tf}(y) =m(y)\hat f(y)$. Then
%$S_\varphi = BTB^{-1}$. We see that
%$$ S^*_\varphi = (B^{-1})^* T^* B^* = BT^* B^{-1} $$
%since $B$ is unitary.
%
%It is not difficult to check that $\widehat{ T^* f}(y) = \overline{m(y)} \hat f(y)$. Thus we have
%$ S^*_\varphi  = S_{\tilde\varphi},   $ where
%$$\tilde \varphi (z) = \int_{\mathbb R^n} \overline{m(x)} e^{ -2(x- {i\over2}z)^2 } dx.$$
%Since  $T$ is normal, so does $S_\varphi$.
%
%The proof of Theorem \ref{th1.3} is complete.
%\end{proof}
%
%\begin{proof}[Proof of Corollary \ref{cor1.4}]
%It is routine. We see that if $m$ is real, then $S_\varphi $ is self-adjoint. If $m$ takes value in  pure imaginary
%numbers, then $S_\varphi$ is anti-selfadjoint, that is, $S_\varphi^* = -S_\varphi$.
%
%This finishes the proof of Corollary \ref{cor1.4}.
%\end{proof}

\begin{rem}\label{rem3.5}
From   \cite[Proposition 2]{Z1}, we know that when $n=1$,
 a necessary condition for $S_{\varphi}$ to be
bounded on ${\mathscr F}^2(\mathbb C)$ is that $\varphi(z- {\bar z})$ is bounded. In other words, the  boundedness of
$S_{\varphi}$ implies that the function $\varphi$ is bounded on the imaginary axis.    However, this is not a sufficient condition, showing that the reproducing kernel thesis fails for this problem.
Indeed, we consider
$$
	\varphi(z) = \int_{\mathbb{R}} \psi(x) e^{-2(x-\frac{i}{2}z)^2} dx,
$$
where $\psi(x)$ belongs to $L^4(\R)\backslash L^\infty(\R)$. 
Following the proof of Lemma~\ref{le2.1}, we  can see $\varphi\in  {\mathscr F}^2(\mathbb C) $, hence $\varphi$ can gives rise to the operator $S_\varphi$. H\"older's inequality shows that $\varphi(z-\bar{z})$ is bounded on the imaginary axis.
But it can not be given by
$$
	\varphi(z) = \int_{\mathbb{R}} m(x) e^{-2(x-\frac{i}{2}z)^2} dx
$$
for any bounded function $m$. If this were possible, then there would exist a bounded function $m$ such that $\varphi$ has the above representation.
Then for all $z$,
$$
	\int_{\mathbb{R}} (\psi(x)-m(x)) e^{-2(x-\frac{i}{2}z)^2} dx = 0.
$$
Set $z=u$ to be an arbitrary real number, then it becomes
$$
	\int_{\mathbb{R}} (\psi(x)-m(x)) e^{-2x^2+2xiu} dx = 0,
$$
which means $\mathcal{F}^{-1}[(\psi(x)-m(x)) e^{-2x^2}](u)=0$. Since $(\psi(x)-m(x)) e^{-2x^2}$ is an $L^2$ function, then we have $\psi(x)=m(x)$,
%
%Then for all $z$,
%$$
%	\int_{\mathbb{R}} (\psi(x)-m(x)) e^{-2(x-\frac{i}{2}z)^2} dx = 0.
%$$
%Differentiating in $z$ and then setting $z=0$, it follows that
%$$
%	\int_{\mathbb{R}} (\psi(x)-m(x)) e^{-2x^2} x^k dx = 0,
%$$
%where $k$ is any positive integer. This easily implies that $\psi(x)=m(x)$,
%
which is a contradiction.
Therefore, by the theorem $S_\varphi$ is not bounded on ${\mathscr F}^2(\mathbb C)$, although $\varphi$ is bounded on the imaginary axis.
\end{rem}

From Theorem \ref{th1.1}, we see that from the multiplier function $m$ we obtain the analytic function $\varphi$. We now show how $\varphi$ gives rise to $m$.
\begin{prop}\label{prop reverse}
 Suppose $\varphi \in\mathscr F^2(\mathbb C^n)$ such that $S_\varphi$ is bounded on $\mathscr F^2(\mathbb C^n)$.  Then for  $Tf:=B^{-1}S_{\varphi}Bf$, $f\in L^2(\mathbb R^n)$, we have
 $\mathcal F (Tf)(x) = m(x)\mathcal F f(x)$ with
$$m(x)=\int_{\mathbb{C}^n} \int_{\mathbb{C}^n}  \varphi(z-\bar{w}) e^{z\cdot\bar{w}-2ix\cdot\bar{z}+\frac{\bar{z}^2}{2}}dwdz.$$
\end{prop}
\begin{proof}
Consider $Tf:=B^{-1}S_{\varphi}Bf$.  Taking the Fourier transform it becomes $m\mathcal{F}f=\mathcal{F}B^{-1}S_{\varphi}Bf$. Let $$f_0(x)=e^{-x^2},$$ then we have $Bf_0(z)=\left({2/ \pi}\right)^{-{n/ 4}}.$ It follows that
$$
	S_{\varphi}Bf_0(z)=\left({2\over \pi}\right)^{-{n\over 4}} \int_{\mathbb{C}^n} e^{z\cdot\bar{w}} \varphi(z-\bar{w}) d\lambda(w).
$$
%Note that
%$$
%	B\mathcal{F}B^{-1}f(z)=f(-iz).
%$$
By Lemma \ref{prop2.2}, we have
$$
	(B\mathcal{F}B^{-1})S_{\varphi}Bf_0(z)=S_{\varphi}Bf_0(-iz)=\left({2\over \pi}\right)^{-{n\over 4}} \int_{\mathbb{C}^n} e^{-iz\cdot\bar{w}} \varphi(-iz-\bar{w}) d\lambda(w).
$$
Now we get
\[
\begin{split}
\mathcal{F}B^{-1}S_{\varphi}Bf_0(x)
&= B^{-1}(B\mathcal{F}B^{-1})S_{\varphi}Bf_0(x)\\
&= \int_{\mathbb{C}^n} \int_{\mathbb{C}^n} e^{-iz\cdot\bar{w}} \varphi(-iz-\bar{w}) e^{-x^2+2x\cdot\bar{z}-\frac{\bar{z}^2}{2}}d\lambda(w)d\lambda(z)\\
&= \int_{\mathbb{C}^n} \int_{\mathbb{C}^n} e^{z\cdot\bar{w}} \varphi(z-\bar{w}) e^{-2ix\cdot\bar{z}+\frac{\bar{z}^2}{2}}d\lambda(w)d\lambda(z)\cdot f_0(x).
\end{split}
\]
Since $\mathcal{F}f_0(x)=f_0(x)$,  we get the relation
$$
	m(x)=\int_{\mathbb{C}^n} \int_{\mathbb{C}^n}  \varphi(z-\bar{w}) e^{z\cdot\bar{w}-2ix\cdot\bar{z}+\frac{\bar{z}^2}{2}}d\lambda(w)d\lambda(z).
$$
The proof of Proposition \ref{prop reverse} is complete.
\end{proof}

As from {the first comment in} Remark \ref{rem1.2},
it is natural to ask whether the characterization of $S_\varphi$ as in Theorem \ref{th1.1} can imply some boundedness on the Fock space
${\mathscr F}^p(\CC)$ for $p\in [1,\infty)$. For $p>2$, for $S_\varphi$ defined in \eqref{e1.1} with $\varphi$ as in \eqref{e1.2}, by using H\"older's inequality
one can verify that
$ S_\varphi $ is bounded from ${\mathscr F}^p(\CC)$ to ${\mathscr F}^{p'}(\CC)$. We omit the details here.
However, this is not true in general when $p\in[1,2)$. We now provide a counterexample in dimension $n=1$ with $S_\varphi=BHB^{-1}$, where $H$ is the Hilbert transform on $\mathbb R$ (we refer to Example 2 in Section \ref{s:Apps} for details, see also \cite[Section 8]{Z3}).
\begin{prop}\label{prop counter}
Let $S_\varphi=BHB^{-1}$, where $H$ is the Hilbert transform on $\mathbb R$. Suppose $1\leq p<2$. Then $S_\varphi$ is not well-defined on $\mathscr F^p(\mathbb C)$.
\end{prop}
\begin{proof}
For $S_\varphi=BHB^{-1}$, we see that from Example 2 in Section \ref{s:Apps}, the function $\varphi$ is as in \eqref{e1.2} with  $m(x):=-i {\rm sgn}(x)$. Consider $F(w):=e^{\frac{w^2}{2}}$. Note that this function $F$ is in  $\mathscr F^p(\mathbb C)$ for all $1\leq p<2$ but is not in $\mathscr F^p(\mathbb C)$ for any $ p\geq2$.
Then
\[
\begin{split}
S_{\varphi}F(z)
&= -i \int_0^\infty \int_{\mathbb{C}} F(w)e^{z\bar{w}}e^{-2(x-\frac{i}{2}(z-\bar{w}))^2} d\lambda(w)dx\\
& \ \ \ \ + i \int_{-\infty}^0 \int_{\mathbb{C}} F(w)e^{z\bar{w}} e^{-2(x-\frac{i}{2}(z-\bar{w}))^2} d\lambda(w)dx \\
&= -i \int_0^\infty \int_{\mathbb{C}} e^{\frac{w^2}{2}}\, e^{z\bar{w}} \left(e^{-2(x-\frac{i}{2}(z-\bar{w}))^2}-e^{-2(x+\frac{i}{2}(z-\bar{w}))^2}\right) d\lambda(w)dx  \\
&= -i \,e^{\frac{z^2}{2}}\,\int_0^\infty \int_{\mathbb{C}} e^{\frac{w^2}{2}+\frac{\bar{w}^2}{2}} (e^{2xi(z-\bar{w})}-e^{-2xi(z-\bar{w})}) d\lambda(w)\ e^{-2x^2} dx .
\end{split}
\]
Now we see that by writing $w=a+ib$,
\[
\begin{split}
\
&\int_{\mathbb{C}} e^{\frac{w^2}{2}+\frac{\bar{w}^2}{2}} (e^{2xi(z-\bar{w})}-e^{-2xi(z-\bar{w})}) d\lambda(w)\\
&\ \ \ \ \ \ \ \ ={1\over \pi}\int_{\mathbb{R}} \int_{\mathbb{R}} (e^{2xiz-2xb}e^{-2xia}-e^{-2xiz+2xb}e^{2xia}) da\ e^{-2b^2} db.
\end{split}
\]
But it is obvious that for each $z\in\mathbb C$, $x\in (0,\infty)$, and $b\in\mathbb R$, the integral
$$
	\int_{\mathbb{R}} (e^{2xiz-2xb}e^{-2xia}-e^{-2xiz+2xb}e^{2xia}) da	
$$
 is not convergent. Thus, we see that $S_\varphi$ is not well-defined on $\mathscr F^p(\mathbb C)$.
\end{proof}

{As from the second comment in Remark \ref{rem1.2}, there is another possible definition the Fock space $F^p$ for $1<p<\infty$, which consists of
all entire functions $f$ on $\mathbb C$ such that the function $f(z) e^{-\frac{|z|^2}{2}}$ belongs to $L^p(\mathbb C)$ with the norm
$$ \|f\|_{F^p} := \bigg( {p\over 2\pi} \int_{\mathbb C} \Big| f(z) e^{- |z|^2/2}\Big|^p dz\bigg)^{1\over p}. $$
It is clear that $F^2$ is the same as $\mathscr F^2$.
See for example \cite{CHH}. The main result in \cite{CHH} states that when $2<p<\infty$, the Bargmann transform maps $L^p(\mathbb R)$ boundedly into 
$F^p$, but NOT onto.  Hence, if we consider the operator $S_\varphi=BHB^{-1}$ as in Proposition \ref{prop counter}, which is well defined on $F^2$, then $S_\varphi$ is not well-defined on 
$F^p$ for $2<p<\infty$.

}

\medskip

In the theory of singular integrals in harmonic analysis, it is well-known (see \cite{DJ, St1}) that the famous ``$T(1)$" theorem of David and  Journ\'e gives necessary and sufficient conditions for   generalized
Calderon-Zygmund operators to    be bounded on $L^2(\RR)$.
We  propose the following open problem
%for a criterion of boundedness for general singular integrals (non-convolution type)
on the Fock space ${\mathscr F}^2(\mathbb C^n)$ (see also Proposition \ref{prop Folland}).

\smallskip

{\bf Open problem}:   Characterize those  entire  functions  $K_T(z,w) $ on $ {\mathbb C}^{2n}$
such that the integral operator
$$
TF(z) = \int_{\mathbb C^n} K_T(z, \bar w) F(w) d\lambda(w), \ \ \ \ \ z\in{\CC}
$$
is bounded on  ${\mathscr F}^2(\mathbb C^n)$.

%For example, if we consider a special case $K_{T}(z,\bar w):=\varphi(w)e^{z\cdot \bar w} $ for some entire function $\varphi \in L^2(\mathbb C^n),$ then $T$ is a Toeplitz operator on $\mathscr F^2(\mathbb C^n).$

\medskip

\section{Applications and Examples of Theorem~\ref{th1.1}}
\setcounter{equation}{0}
\label{s:Apps}

%In this section, we study some well-known examples of operators under the Bargmann transform
%from $L^2(\mathbb R^n)$ to $\mathscr F^2(\mathbb C^n)$.

%As pointed out in \cite{Z1},
There are many examples to show that characterising the boundedness of $S_\varphi$ is interesting and non-trivial. By choosing different functions $\varphi$ in $S_\varphi$, one can recover important operators arising in complex analysis and harmonic analysis.
%
%
%For example,
%recall that the Hilbert transform is the operator on $L^2(\R)$ defined by
% $$
% H(f)(x) ={1\over \pi} \int_{\R} {f(y)\over x-y} dy,
% $$
% where the improper integral is taken in the sense of ``principle value".   The Hilbert transform  $H$ is  bounded
% on $L^2(\R)$. It is shown in \cite[Theorem 1]{Z1} that
%\begin{eqnarray}\label{e2.3}
%  BHB^{-1} f(z) ={4\over \pi} \int_{\mathbb{C}} f(w) e^{z \bar{w}} \varphi(z- \bar{w}) d\lambda(w): ={4\over \pi} S_{\varphi}
%\end{eqnarray}
%with
%$$\varphi(z)=-{2\over \sqrt{\pi}} \int_0^{2^{-1/2}z} e^{u^2} du.
%$$
%Here $B$ is
% the Bargmann transform, given by
% \begin{eqnarray}\label{e2.1}
% Bf(z)&=& \Big({2\over \pi}\Big)^{1\over 4}\int_{\R} f(x) e^{2xz-x^2- {z^2\over 2}} dx
%\end{eqnarray}
%and its inverse $B^{-1}.$
%From \eqref{e1.2}, it is concluded that  the classical Hilbert transform is unitarily equivalent to
% an operator of the form $S_{\varphi}$ on the Fock space ${\mathscr F}^2(\CC)$. For this reason,  we call the operator
%   \begin{eqnarray*}
%	S_{\varphi}F(z)= \int_{\mathbb{C}} F(w) e^{z \bar{w}} \varphi(z- \bar{w}) d\lambda(w)
%\end{eqnarray*}
% a singular integral operator.
%
%
%
%
%
%
We now apply our main result Theorem \ref{th1.1} to a few well-known examples,
such as the Riesz transform on $\mathbb R^n$, the Ahlfors--Beurling operator on $\mathbb C$, and a few others.

\medskip
\noindent
 {\bf Example 1}.  If $S_\varphi$ is the identity with $\varphi(z)=1$, then $\varphi$ can be written as \eqref{e1.2} , where $m(x)=1.$

\medskip

\noindent
{\bf Example 2}.  Let $S_\varphi=BHB^{-1}$ with $H$ the Hilbert transform defined as
 $$
 H(f)(x) ={\rm p.v.}{1\over \pi} \int_{\R} {f(y)\over x-y} dy,
 $$
 where the improper integral is taken in the sense of ``principle value."
 Note that $\mathcal F (Hf)(x) = m(x)\mathcal F f(x) $ with $m(x)=-i{\rm sgn}(x)$.

By Theorem~\ref{th1.1},
 the function $\varphi$ can be written as \eqref{e1.2} with  $m(x)=-i {\rm sgn}(x).$ That is,
\[ \varphi(z)=\left({2\over \pi}\right)^{1\over 2} \int_{\mathbb{R}} -i {\rm sgn}(x)e^{-2(x-\frac{i}{2}z)^2} dx .   \]
Note that
\begin{eqnarray*}
\frac{d}{dz}\varphi(z)
&=&-i\left({2\over \pi}\right)^{1\over2} \left(\int_0^\infty-\int_{-\infty}^0\right) \left(-4(x-\frac{i}{2}z)\right) \left (-\frac{i}{2}\right) e^{-2(x-\frac{i}{2}z)^2} dx\\
&=&-\left(2 \pi \right)^{-{1\over2}} e^{-2(x-\frac{i}{2}z)^2}\left(\Big{|}_0^\infty-\Big{|}_{-\infty}^0\right)
 =\left({2\over \pi}\right)^{1/2} e^{z^2\over2}
\end{eqnarray*}
with $   \varphi(0)=0.$ This implies
$$
	\varphi(z)={2\over\sqrt{\pi}}  A\left(\frac{z}{\sqrt{2}}\right)\in {\mathscr F}^2(\C),
$$
where
$$
	A(z)=\int_0^z e^{u^2} du, \ \ \ \ z\in \C,
$$
which is the antiderivative of $e^{u^2} $ satisfying $A(0)=0$. See also \cite[Section 8]{Z3}.
\medskip

\noindent
{\bf Example 3}.  From \cite{Z1}, if $\varphi(z)=e^{az^2}$ with $0<a<\frac 12$, the operator  $S_\varphi$ is bounded on ${\mathscr F}^2(\C)$.  By Theorem~\ref{th1.1}, $\varphi$ can be written as \eqref{e1.2} for some $m\in L^\infty (\mathbb{R})$, hence
$$
	\int_{\mathbb{R}} m(x)e^{-2(x-\frac{i}{2}z)^2-az^2} dx = \int_{\mathbb{R}} m(x)e^{-\left(\frac{x}{\sqrt{\frac 12 -a}}-i\sqrt{\frac 12 -a}z\right)^2} e^{\frac{4a}{1-2a}x^2} dx
$$
should be a constant. Thus we are able to choose $m(x)=e^{-\frac{4a}{1-2a}x^2}$, which is a bounded function.

\medskip

\noindent
{\bf Example 4}.  Let $\varphi(z)=e^{z\bar{a}}$. If $\varphi(z)$ has the representation \eqref{e1.2}, then
$$
	\int_{\mathbb{R}} m(x)e^{-2x^2+(2ix-\bar{a})z+\frac 12 z^2} dx = \int_{\mathbb{R}} m(x)e^{2xi\bar{a}-\frac{\bar{a}^2}{2}} e^{-2(x+\frac{i}{2}(\bar{a}-z))^2} dx
$$
should be a constant, hence
$$
	\int_{\mathbb{R}} (m(x)e^{2xi\bar{a}-\frac{\bar{a}^2}{2}}-c) e^{-2(x+\frac{i}{2}(\bar{a}-z))^2} dx=0
$$
for some constant $c$.

Thus $m(x)=c_0 e^{-2xi\bar{a}}$ almost everywhere, where $c_0$ is a constant. By Theorem~\ref{th1.1}, $S_\varphi$ is bounded on ${\mathscr F}^2(\CC)$ if and only if $m$ is bounded, i.e. $a$ is real. In fact, this is a result shown
in \cite{Z1} and when $a$ is real, $S_\varphi=W_a$, which is a unitary operator defined above.

%(a) $e^{z\cdot \bar{w}}$ is the reproducing kernel.
%
%(b) $B\mathcal{F}B^{-1}F(z)=F(-iz)$, where $\mathcal{F}$ is the Fourier transform on $\mathbb{R}^n$, which is given by
%$$
%	\mathcal{F}f(x)=\pi^{-\frac{n}{2}} \int_{\mathbb{R}^n} f(t)e^{-2ix\cdot t}dt.
%$$
%And the inverse
%$$
%	\mathcal{F}^{-1}F(x)=\pi^{-\frac{n}{2}} \int_{\mathbb{R}^n} f(t)e^{2ix\cdot t}dt.
%$$

\medskip

\noindent
{\bf Example 5}. Riesz transforms on $\mathbb R^n$.

We now recall the
Riesz transform on $\mathbb R^n$: for $f\in L^2(\mathbb R^n)$, $x\in \mathbb R^n$, $ j=1,2,\ldots,n,$ the $j$-th Riesz transform is defined as
$$ R_j f(x) =\lim_{\epsilon\to0}  \int_{|y|\geq\epsilon} K_j(y) f(x-y) dy, \quad$$
where
$$ K_j(y)= c_n {y_j\over |y|^{n+1}},\quad c_n ={\Gamma({n+1\over2}) \over \pi^{n+1\over2}}. $$

Note that for $j=1,2,\ldots,n,$
$$ \mathcal F (R_j f ) (\xi)   = m_j(\xi) \mathcal F(f)(\xi), $$
where
\begin{align}\label{Riesz multiplier}
m_j(\xi):=-i {\xi_j\over |\xi|}.
\end{align}
Hence we have
$$ R_j(f) (x) = \mathcal F^{-1} \Big(-i {\xi_j\over |\xi|} \mathcal F(f)(\cdot)\Big)(x). $$

Then, by applying our main result Theorem \ref{th1.1} and Lemma \ref{prop2.2},
we obtain that
\begin{prop}\label{th2.3}
For $j=1,2,\ldots,n$, the operator $T_j = B R_j B^{-1}: \mathscr F^2(\mathbb C^n) \to \mathscr F^2(\mathbb C^n)$
is given by
$$ T_jF(z) = \int_{\CC} F(w)e^{z\cdot\bar w} \ \varphi_j(z-\bar w) d\lambda(w) $$
for all $F\in \mathscr F^2(\mathbb C^n)$, with
$$\varphi_j(z):= \left( {2\over\pi}\right)^{n\over2} \int_{\RR}  m_j(\xi)\  e^{-2 (\xi- {iz \over2})^2}d\xi,\quad {\rm where\ } m_j(\xi)=-i\ {\xi_j\over |\xi|}. $$
\end{prop}

From the Fourier multiplier of Riesz transform as given in \eqref{Riesz multiplier}, we see that
%These are given  by
% $$
% R_j(f)(x) = c_n\int_{\RR} {y_j\over |y|^{n+1}} f(x-y) dy, \ \ \ \ j=1, \cdots, n,
% $$
% where $c_n={\Gamma\left({n+1\over 2}\right)\over \pi^{(n+1)/2}}$, and  the improper integral is taken in the sense of ``principle value".   The Riesz transform  $R_j$ is  bounded
% on $L^2(\RR)$.
%  whose Fourier transform is
%  $$m_j(x)=-i\frac{x_j}{|x|}, \ \ \ \ \ j=1,2,\dots,n.
%  $$
 $
	\sum_{j=1}^n m_j^2(\xi) + 1 = 0,
$
which  gives a fundamental equation for Riesz transforms:
\begin{eqnarray}\label{e3.6}
	\sum\limits_{j=1}^n R_j^2 =- Id,
\end{eqnarray}
where $Id$ is the identity operator on $L^2(\mathbb R^n)$.

Define the operators $S_{\varphi_j}$ by
\begin{eqnarray}\label{e3.77}
S_{\varphi_j}= BR_jB^{-1}, \ \ \ \ \ j=1,2,\dots,n.
\end{eqnarray}

\begin{prop}\label{th3.4}
 The following  equation holds for  the operators $\{S_{\varphi_j}\}$
\begin{eqnarray}\label{e3.7}
	\sum\limits_{j=1}^n S_{\varphi_j}^2 =- {\bf Id}.
\end{eqnarray}
with
 $
	\sum\limits_{j=1}^n ||\varphi_j||_{{\mathscr F}^2(\CC)}^2 = 1,
 $ where ${\bf Id}$ is the identity operator on $\mathscr F^2(\mathbb C^n)$.
\end{prop}

\begin{proof}  Note that  $m_j$ is an odd function, so is $\varphi_j$.  Write
\begin{eqnarray*}
	S_{\varphi_{jj}}F(z) &=&BR_j^2B^{-1}F(z)
=\int_{\mathbb{C}^n} F(\xi) e^{z\cdot \bar{\xi}} \varphi_{jj}(z-\bar{\xi}) d\lambda(\xi).
\end{eqnarray*}
On the other hand
\begin{eqnarray*}
S_{\varphi_{jj}}F(z)
&=& (BR_jB^{-1})(BR_jB^{-1})F(z)
= \int_{\mathbb{C}^n} F(\xi)
\left(\int_{\mathbb{C}^n} \varphi_j(z-\bar{w}) \varphi_j(w-\bar{\xi}) e^{w\cdot \bar{\xi}} e^{z\cdot \bar{w}} d\lambda(w)\right) d\lambda(\xi).
\end{eqnarray*}
Since $F$ is arbitrary,  we get
\begin{eqnarray}\label{ww}
	e^{z\cdot \bar{\xi}} \varphi_{jj}(z-\bar{\xi}) = \int_{\mathbb{C}^n} \varphi_j(z-\bar{w}) \varphi_j(w-\bar{\xi}) e^{w\cdot \bar{\xi}} e^{z\cdot \bar{w}} d\lambda(w).
\end{eqnarray}
Set $z=\xi$ and notice $\overline{\varphi_j(z)}=\varphi_j(\bar{z})$, then it follows that
\begin{eqnarray*}
\varphi_{jj}(z-\bar{z})
&=& \int_{\mathbb{C}^n} \varphi_j(z-\bar{w}) \varphi_j(w-\bar{z}) e^{w\cdot \bar{z}} e^{z\cdot \bar{w}} e^{-|z|^2} d\lambda(w)\\
&=& -\pi^{-n} \int_{\mathbb{C}^n} |\varphi_j(z-\bar{w})|^2 e^{w\cdot \bar{z}} e^{z\cdot \bar{w}} e^{-|z|^2-|w|^2} dw\\
&=& -\pi^{-n} \int_{\mathbb{C}^n} |\varphi_j(z-\bar{w})|^2 e^{-|z-w|^2} dw\\
&=& -\int_{\mathbb{C}^n} |\varphi_j(w+z-\bar{z})|^2 d\lambda(w).
\end{eqnarray*}
However,
\begin{eqnarray*}
\sum\limits_{j=1}^n \varphi_{jj}(z-\bar{z})
&=& \sum\limits_{j=1}^n \left( {2\over\pi}\right)^{n\over2} \int_{\mathbb{R}^n} m_j^2(x) e^{-2(x-\frac{i(z-\bar{z})}{2})^2} dx
= -\left( {2\over\pi}\right)^{n\over2} \int_{\mathbb{R}^n} e^{-2(x-\frac{i(z-\bar{z})}{2})^2} dx
= -1.
\end{eqnarray*}
Then it implies
$$
	\sum\limits_{j=1}^n \int_{\mathbb{C}^n} |\varphi_j(w+it)|^2 d\lambda(w) = 1.
$$
Define translation along the imaginary axis $\tau_t f(z) = f(z+it)$, where $t$ is real. Then it says the sum
$$
	\sum\limits_{j=1}^n ||\tau_t\varphi_j||_{{\mathscr F}^2(\CC)}^2 = 1
$$
under any translation along the imaginary axis.
In particular,  we have that
 $
	\sum\limits_{j=1}^n ||\varphi_j||_{{\mathscr F}^2(\CC)}^2 = 1.
 $
 Moreover, we set  $\xi=0$ in  \eqref{ww}, then we get
$$
	\varphi_{jj}(z) = S_{\varphi_j}(\varphi_j)(z), \quad z\in\mathbb C^n,
$$
hence
$$
	\sum\limits_{j=1}^n S_{\varphi_j}(\varphi_j)(z) + 1 =0.
$$
 The proof of Proposition~\ref{th3.4} is complete.
\end{proof}

\medskip

\noindent
{\bf Example 6}. Ahlfors--Beurling operator on $\mathbb C$.

The Ahlfors--Beurling operator is a very well-known Calder\'on--Zygmund operator on
$\mathbb C$, defined on $L^p(\mathbb C)$, $1<p<\infty$, as follows:
$$ \mathscr B \psi(z) = {\rm p.v.}\ {1\over \pi} \int_{\mathbb C} \frac{ \psi(\xi)}{(\xi- z)^2 }d\xi. $$
It connects harmonic analysis and complex analysis and is of fundamental importance in several areas of mathematics including PDE and quasiconformal mappings. For example, Petermichl and Volberg \cite{PV} proved a sharp weighted  estimate of $\mathscr B$,
which shows that any weakly quasiregular map is quasiregular. We also recall that $\mathscr B$ is an isometry on $L^2(\mathbb C)$,
and is given as a Fourier multiplier
of $\mathcal{F}(\mathscr Bf)(\xi)=m(\xi) \mathcal{F}(f)(\xi)$, where
$$m(\xi) ={ \bar \xi\over\xi},\quad \xi\in\mathbb C.$$

Then by applying Theorem \ref{th1.1}, we have the following.
\begin{prop}\label{prop Beurling}
The operator $T = B \mathscr B B^{-1}: \mathscr F^2(\mathbb C^2) \to \mathscr F^2(\mathbb C^2)$
is given by
$$ TF(z) = \int_{\mathbb C^2} F(w)e^{z\cdot\bar w} \ \varphi(z-\bar w) d\lambda(w) $$
for all $F\in \mathscr F^2(\mathbb C^2)$, with
$$\varphi(z-\bar w):=  {2\over\pi} \int_{\mathbb R^2}  m(x)\  e^{-2 (x- {i(z-\bar w) \over2})^2}dx,\quad {\rm where\ } m(x)=\left( {x_1-ix_2\over x_1+ix_2}\right), \quad x=(x_1,x_2)\in\mathbb R^2 . $$
\end{prop}

\begin{proof}
For every  $F\in \mathscr F^2(\mathbb C^n)$, we have
\begin{align*}
T F(z)&= B \mathscr B B^{-1} F(z)= B\ \mathcal F^{-1} \Big( {\bar \xi\over \xi}\Big)\ \mathcal F\ \big(B^{-1} F \big) (z)\\
&= B \mathcal F^{-1} B^{-1}\ \bigg[B\ \Big( {\bar\xi\over \xi}\Big)\ \mathcal F\big(B^{-1} F \big) \bigg](z)\\
&=  B\ \Big( {\bar\xi\over \xi}\Big)\ \mathcal F\big(B^{-1} F \big) (iz),
\end{align*}
where the last equality follows from Proposition \ref{prop2.2}.

Then from the definition of the Bargmann transform and from \eqref{F-1BF}, we have
\begin{align*}
T F(z)
&=\Big( {2\over\pi}\Big)^{1\over2}\int_{\mathbb R^2}  \Big( {x_1-ix_2\over x_1+ix_2}\Big)\ \mathcal F\big(B^{-1} F \big)(x)\ e^{2x\cdot (iz)-x^2-{(iz)^2\over2}} dx\\
&={2\over\pi}\int_{\mathbb R^2}   \Big( {x_1-ix_2\over x_1+ix_2}\Big)   e^{-x^2} \int_{\mathbb C^2} F(w)e^{({\bar w}^2/2)}   e^{-2i\bar w \cdot x}  d\lambda(w)\ e^{2x\cdot (iz)-x^2-((iz)^2/2)} dx\\
%&={2\over\pi}\int_{\mathbb C^2} F(w)e^{({\bar w}^2/2)} \ \int_{\mathbb R^2} \Big( {x_1-ix_2\over x_1+x_2}\Big) \   e^{-x^2}   e^{-2i\bar w \cdot x}  \ e^{2x\cdot (iz)-x^2-((iz)^2/2)} dx\ d\lambda(w)\\
%&= {2\over\pi}\int_{\mathbb C^2} F(w)e^{{{\bar w}^2\over2}} \ \int_{\mathbb R^2}  \Big( {x_1-ix_2\over x_1+x_2}\Big) \  e^{-2 (x+ {i(z-\bar w) \over2})^2} e^{- { (z-\bar w)^2\over2} + {z^2\over2}}   dx\ d\lambda(w)\\
&= {2\over\pi}\int_{\mathbb C^2} F(w)e^{z\cdot\bar w} \ \int_{\mathbb R^2} \Big( {x_1-ix_2\over x_1+ix_2}\Big) \  e^{-2 (x- {i(z-\bar w) \over2})^2}    dx\ d\lambda(w)\\
&=\int_{\mathbb C^2} F(w)e^{z\cdot\bar w} \ \varphi(z-\bar w) d\lambda(w).
\end{align*}
The proof of Proposition \ref{prop Beurling} is complete.
\end{proof}

Parallel to the powers of Riesz transform (Proposition \ref{th3.4}), we also have the following
direct result of the powers of the Ahlfors--Beurling operator (see for example \cite{DV}).

\begin{cor}\label{cor Beurling}
Suppose $k$ is a positive integer and $k>1$. The operator $T^k = B \mathscr B^k B^{-1}: \mathscr F^2(\mathbb C^2) \to \mathscr F^2(\mathbb C^2)$
is given by
$$ T^kF(z) = \int_{\mathbb C^2} F(w)e^{z\cdot\bar w} \ \varphi_k(z-\bar w) d\lambda(w) $$
for all $F\in \mathscr F^2(\mathbb C^2)$, with
$$\varphi_k(z-\bar w):=  {2\over\pi} \int_{\mathbb R^2}  m_k(x)\  e^{-2 (x- {i(z-\bar w) \over2})^2}dx,\quad {\rm where\ } m_k(x)=\left( {x_1-ix_2\over x_1+ix_2}\right)^k, \quad x=(x_1,x_2)\in\mathbb R^2 . $$
\end{cor}

%\bigskip
%\bigskip
%\bigskip
%\bigskip

 \medskip

\section{Operator Theoretic Properties of the Operator $S_\varphi$}
\label{s:Operator}

In this section we study operator theoretic properties of the singular integral operator $S_\varphi$.
In particular, we are able to determine the normality, $C^*$ algebraic properties, compactness, and
the   spectrum of the operator $S_\varphi$. Moreover, we also obtain the reducing  subspaces of $S_{\varphi}$.

%and then  present some examples as  applications of our Theorem~\ref{th1.1}.

%\medskip\noindent
%{{\bf2}: Conditions on $m$ so that $S_{\varphi}$ is self-adjoint or normal}

\subsection{Normality of $S_\varphi$: Proof of  Theorem \ref{th1.3}}

\begin{proof}[Proof of Theorem \ref{th1.3}]

%From Theorem \ref{th1.3}, we can obtain more
%property of the singular integral operator
%$$S_\varphi f(z) = \int_{\mathbb C^n} f(w) e^{z\cdot \bar w} \varphi(z-\bar w) d\lambda(w).$$

For any $f,g\in {\mathscr F}^2(\mathbb C^n)$,
\begin{align*}
\langle  S^*_\varphi f,g \rangle_{\mathscr F^2(\mathbb C^n)}
&= \langle   f, S_\varphi g \rangle_{\mathscr F^2(\mathbb C^n)} =  \int_{\mathbb C^n} f(z) \overline{S_\varphi g(z) }d\lambda(z)\\
&=  \int_{\mathbb C^n} f(z) \overline{ \int_{\mathbb C^n} g(w) e^{z\cdot \bar w} \varphi(z-\bar w) d\lambda(w) }d\lambda(z)\\
&=  \int_{\mathbb C^n} f(z)  \int_{\mathbb C^n} \bar g(w) e^{\bar z\cdot  w} \overline{\varphi(z-\bar w) }d\lambda(w) d\lambda(z).
\end{align*}
Note that by Theorem~\ref{th1.1},
\begin{align*}
 \varphi(z-\bar w) =\left({2\over \pi}\right)^{n\over 2} \int_{\mathbb R^n} m(x) e^{-2 (x - {i\over2} (z-\bar w))^2}dx
\end{align*}
for some $L^{\infty}(\RR)$ function $m$ such  that
\begin{align*}
\tilde \varphi(w-\bar z)& :=\overline{ \varphi(z-\bar w)}=\left({2\over \pi}\right)^{n\over 2} \int_{\mathbb R^n} \overline{m(x)} e^{-2 (x + {i\over2} (\bar z- w))^2}dx =\left({2\over \pi}\right)^{n\over 2} \int_{\mathbb R^n} \overline{m(x)} e^{-2 (x - {i\over2} ( w- \bar z))^2 }dx.
\end{align*}
Thus, by Fubini's theorem,
\begin{align*}
\langle  S^*_\varphi f,g \rangle_{\mathscr F^2(\mathbb C^n)}
&=  \int_{\mathbb C^n} \int_{\mathbb C^n} f(z) \ \bar g(w)\ e^{\bar z\cdot  w}\ {\overline {\varphi(z-\bar w)}} d\lambda(z) d\lambda(w)\\
&=\int_{\mathbb C^n} \int_{\mathbb C^n} f(z)  \ e^{ w\cdot \bar z}\ \tilde \varphi(w-\bar z)\ d\lambda(z)\ \bar g(w)\ d\lambda(w).
\end{align*}

Hence, we have
\begin{align*}
S^*_\varphi f(z)
&= \int_{\mathbb C^n} f(w)  \ e^{ z\cdot \bar w}\ \tilde \varphi(z-\bar w)\ d\lambda(w)=: S_{\tilde\varphi}f(z).
\end{align*}

By noting that $S_\varphi S_{\psi}=S_{\psi} S_\varphi $ for any bounded operators $S_{\psi}$ and $S_{\varphi}$, we see that $S_{\varphi}$ is always normal.
%
%And
%\begin{align*}
%S_\varphi(S^*_\varphi f)(z)
%&= \int_{\mathbb C^n} (S^*_\varphi f)(w)  \ e^{ z\cdot \bar w}\  \varphi(z-\bar w)\ d\lambda(w)\\
%&= \int_{\mathbb C^n} \int_{\mathbb C^n} f(u) e^{w\cdot \bar u} \tilde \varphi(w-\bar u) d\lambda(u) \ e^{ z\cdot \bar w}\  \varphi(z-\bar w)\ d\lambda(w)\\
%&= \int_{\mathbb C^n} \int_{\mathbb C^n} f(u) e^{w\cdot \bar u} \tilde \varphi(w-\bar u) d\lambda(u) \ e^{ z\cdot \bar w}\  \varphi(z-\bar w)\ d\lambda(w)\\
%&= \int_{\mathbb C^n} \int_{\mathbb C^n} f(u) e^{w\cdot \bar u+z\cdot \bar w} \tilde \varphi(w-\bar u)  \  \varphi(z-\bar w)\ d\lambda(w)\ d\lambda(u).
%\end{align*}
%Moreover,
%\begin{align*}
%S^*_\varphi (S_\varphi f)(z)
%&= \int_{\mathbb C^n} (S_\varphi f)(w)  \ e^{ z\cdot \bar w}\ \tilde \varphi(z-\bar w)\ d\lambda(w)\\
%&= \int_{\mathbb C^n}\int_{\mathbb C^n} f(u) \ e^{w\cdot \bar u}\ \varphi(w-\bar u) \ e^{ z\cdot \bar w}\ \tilde \varphi(z-\bar w)\ d\lambda(w)d\lambda(u)\\
%&= \int_{\mathbb C^n}\int_{\mathbb C^n} f(u) \ e^{w\cdot \bar u+z\cdot \bar w}\ \varphi(w-\bar u) \ \tilde \varphi(z-\bar w)\ d\lambda(w)d\lambda(u).
%\end{align*}
%
%Recall that
%$$ \varphi(w-\bar u) =\int_{\mathbb R^n} m(x) e^{-2( x- {i\over 2}(w-\bar u) )^2}dx$$
%and
%$$ \tilde \varphi(z-\bar w) =\int_{\mathbb R^n} \overline{m(x)} e^{-2( x- {i\over 2}(z-\bar w) )^2}dx.$$
%
%
This finishes the proof of Theorem \ref{th1.3}.
\end{proof}

\medskip

%\medskip\noindent
%{{\bf3}: $S_\varphi$ forms an algebra}

\subsection{$C^*$-Algebra Generated by $S_\varphi$, Spectrum and Compactness of the Operator $S_\varphi$}

As applications of Theorems \ref{th1.1} and \ref{th1.3}, we can now figure out the $C^*$-algebra,
 the spectrum and the compactness of the operator $S_\varphi$, which were all unknown before. 
 This in turn shows the importance of our Theorem \ref{th1.1}.  
 Here and in what follows, we denote by $M_m f= m\cdot f$  the multiplication operator
  $ M_m$ on $L^2(\mathbb R^n)$ for a function $m$ in $L^\infty(\mathbb R^n)$.

\medskip
\noindent {\bf 5.2.1.\ $C^*$-Algebra Generated by $S_\varphi$}.\
We first have the following result.

\begin{thm}\label{prop4.1}
$\mathscr A:=\{ S_\varphi:\ S_\varphi \ {\rm is\ bounded\ on\ } \mathscr F^2(\mathbb C^n) \}$ is a commutative $C^*$-algebra.
\end{thm}

\begin{proof}
By Theorem \ref{th1.1}, we know that for any $\varphi\in \mathscr F^2(\mathbb C^n),$ $S_\varphi$ is bounded if and only if there is an $m\in L^\infty(\mathbb R^n)$ such that
\eqref{e1.2} holds, and thus $S_\varphi = BTB^{-1}$, where $T\in\mathscr M^{2,2}(\mathbb R^n)$ with
$\mathcal{F}(Tf)(y)=m(y)\mathcal{F}f(y)$.

Hence, we have $ S_\varphi(f)(z) = B \mathcal F^{-1} M_m \mathcal F B^{-1} f(z) $, where $M_mf = m\cdot f$ for
 $f\in L^2(\mathbb R^n)$. If $\varphi_1$ and $\varphi_2$ are in 
 ${\mathscr F}^2(\mathbb C^n)$ such that both $S_{\varphi_1}$ and $S_{\varphi_2}$ are bounded, 
 then there are $m_1$ and $m_2$ in $L^\infty(\mathbb R^n) $ such that
 $S_{\varphi_1} f =B \mathcal F^{-1} M_{m_1} \mathcal F B^{-1} f $ and $S_{\varphi_2} f =B \mathcal F^{-1} M_{m_2} \mathcal F B^{-1} f.$

Furthermore,
\begin{align*}
S_{\varphi_1}S_{\varphi_2}f &= B \mathcal F^{-1} M_{m_1} \mathcal F B^{-1} \big(B \mathcal F^{-1} M_{m_2} \mathcal F B^{-1} f \big)\\
&=B \mathcal F^{-1} M_{m_1}M_{m_2} \mathcal F B^{-1} f \\
&=B \mathcal F^{-1} M_{m_1\cdot m_2} \mathcal F B^{-1} f,
\end{align*}
which shows that
$S_{\varphi_1}S_{\varphi_2} = S_{\varphi}$, where
$$ \varphi(z) = \left({2\over \pi}\right)^{n\over2}\int_{\mathbb R^n} m_1(x)m_2(x) e^{  -2( x- {i\over2}z )^2 } dx.$$

This shows that $ \mathscr A$
%$$ \mathscr A :=\{ S_\varphi: S_\varphi {\rm\ is\ bounded\ on\ } \mathcal F_1 \} $$
is an algebra on $\mathscr F^2 (\mathbb C^n) $.
 Since $S^*_\varphi = S_{\tilde\varphi}$,  and $ S_\varphi S_\psi=S_\psi S_\varphi $
for any $S_\varphi,S_\psi\in \mathscr A$, we see that
$\mathscr A$ is a commutative $C^*$- algebra.  In fact,
$$\mathscr A \cong L^\infty(\mathbb R^n)$$
with the isomorphism map $\mathfrak h: S_\varphi\to m$ for
$$ \varphi(z) = \left({2\over \pi}\right)^{n\over2}\int_{\mathbb R^n} m(x) e^{  -2( x- {i\over2}z )^2 } dx.$$
This finishes the proof of Theorem \ref{prop4.1}.
\end{proof}

\begin{rem}
%{\color{red}Note that $L^{\infty}(\RR)$ is a maximal commutative $w^{\ast}$-algebra, also is $\mathscr A$}.
Note that $L^{\infty}(\RR)$ is a maximal commutative $w^{\ast}$-algebra in $L^2(\RR)$ (see for example \cite[Theorem 4.58]{Do1}). Moreover, from the proof of Theorem
 \ref{prop4.1} we see that $\mathscr A \cong L^\infty(\mathbb R^n)$. Hence, we get that
 $\mathscr A$ is also a maximal commutative $w^{\ast}$-algebra in $\mathscr F^2 (\mathbb C^n)$.
 Thus,  for any bounded
linear operator $T$ on $\mathscr F^2 (\mathbb C^n) $, $T\in  {\mathscr A}$ if and only if
$T S_{\varphi}=  S_{\varphi} T$ for any $ S_{\varphi}\in  {\mathscr A}$. It should be pointed out that   ${\mathscr A}$ has zero factors, in fact, if $m_1,m_2\in L^{\infty }(\mathbb R^n)$ satisfy
$
\left\vert \operatorname{supp} m_1\cap \operatorname{supp} m_2\right\vert=0,
$
then $S_{\varphi_{1}}S_{\varphi_{2}}=0$, where  ${\varphi_{1}}$ and $ {\varphi_{2}}$ are defined as in \eqref{e1.2} for $m_1,m_2.$

One may be concerned that the result in \cite[Theorem 4.58]{Do1} is for a compact Hausdorff space $X$ while we applied it for $X=\mathbb R^n$, which is not compact. However, in this case, all we need to do is first to apply it on a large fixed ball centered at the origin with radius $k$ in $\mathbb R^n$ and then pass to $\mathbb R^n$ by letting $k\to\infty$.  We omit the details here.
\end{rem}

\begin{rem}
If $m(x)$ is a real-valued function, then $\varphi=\tilde\varphi$. Thus, $S^*_\varphi =S_\varphi $, that is,
$S_\varphi $ is self-adjoint. If  $m(x)$ is the function taking purely imaginary values, then
$\tilde\varphi=-\varphi$.
%Thus, $S^*_\varphi S_\varphi = S_\varphi  S^*_\varphi  $, that is, $S_\varphi $ is a normal operator on $\mathscr F^2(\mathbb C^n)$.
Thus, $S^*_\varphi =-S_\varphi $, that is, $S_\varphi$ is anti self-adjoint.  For example, if $S_\varphi = BHB^{-1}$, then $S_\varphi $ is anti self-adjoint.
%In fact,
%$S^*_\varphi = -S_\varphi  $.
\end{rem}

\medskip
\noindent {\bf 5.2.2.\ Spectrum of the operator $S_\varphi$}.\
The computation of the spectrum of an operator $T$ is usually a difficult problem even if $T$ is normal (which our $S_\varphi$ are). But, in this particular case, using the connection with the Fourier multipliers it is possible to rather easily compute the spectrum of $\sigma(S_{\varphi})$ in a  very concise way. Perhaps the proofs of the results in this section are very difficult if one resorts to methods of analytic function theory.
In general, a normal operator may have different spectrum and essential spectrum since the spectrum may contain isolated eigenvalues with finite multiplicity. However, for $\varphi \in \mathscr F^2(\mathbb C^n)$, if $S_{\varphi }$ is bounded, we can prove that the spectrums coincide. Moreover, we also study the eigenvalue of $S_{\varphi }$, as well as the approximate point spectrum.

\begin{thm}\label{th spec}
Suppose  $\varphi\in \mathscr F^2(\mathbb C^n)$ such that $S_\varphi$ is bounded on $\mathscr F^2(\mathbb C^n)$ and $\varphi$ is defined as in \eqref{e1.2} for some $m\in L^\infty(\mathbb R^n).$
Then we have
\begin{itemize}
\item[(1)] $\sigma(S_\varphi) = \mathcal R(m)(\mathbb R^n)$, where $\mathcal R(m)(\mathbb R^n)$
 is the essential range of $m$;
\item[(2)] $\mu\in \mathcal R(m)(\mathbb R^n)$ is an eigenvalue of $S_\varphi$ if and only if
$ \left\vert \{x: m(x)=\mu\} \right\vert>0;  $
\item[(3)]
$\sigma(S_\varphi)=\sigma_a(S_\varphi)$, where $\sigma_a(S_\varphi)$ denotes the approximate point spectrum of $S_\varphi$;
\item[(4)]  $\sigma(S_{\varphi })=\sigma_e(S_{\varphi })$, where $\sigma_e(S_\varphi)$ denotes the essential spectrum of $S_\varphi$.
\end{itemize}
\end{thm}
\begin{proof}
We now provide the proof for these four statements.

\medskip
Proof of (1): this argument is routine by the isomorphism $\mathfrak h:\ S_\varphi\to m$.

\medskip

Proof of (2): for any $\mu\in \mathcal R(m)(\mathbb R^n)$, if $ \left\vert \{x: m(x)=\mu\} \right\vert>0,  $ then
write $\chi_\mu(x) = \chi_{\{x\,: \,m(x)=\mu\}  }(x)$. Without loss of generality, assume
$\left\vert \{x: m(x)=\mu\}\right\vert<\infty$. Then $(M_m-\mu) \chi_\mu =0$ and
$$ \int_{\mathbb R^n}\chi_\mu dx = \left\vert \{x: m(x)=\mu\}\right\vert>0. $$
This shows that $\mu\in \sigma_p({M}_m)$, further $\mu\in \sigma(S_\varphi)$.

On the other hand, if $ \left\vert \{x: m(x)=\mu\} \right\vert=0,$ we can prove that $\mu\not\in \sigma_p(M_m)$.
In fact, for any $f\in L^2(\mathbb R^n)$, if ${ M}_{\mu} f = \mu f$, then $f=0$ on $\mathbb R^n\backslash  \{x: m(x)=\mu\}$. 
 Hence, $f=0$ a.e. since $ \left\vert  \{x: m(x)=\mu\} \right\vert=0$. Thus $\mu\not\in \sigma_p({ M}_m)$, 
 and consequently $\mu\not\in \sigma_p(S_\varphi)$.

\medskip

Proof of (3): for any $m\in L^\infty(\mathbb R^n)$, write ${ M}_mf = m\cdot f$, for every $f\in L^2(\mathbb R^n)$.

Assume $\mu\in\mathcal R(m)(\mathbb R^n)$, the essential range of $m$.  Then $\left\vert \{ x: |m(x)-\mu|<\epsilon \} \right\vert>0$ for any $\epsilon>0$.
Let $\chi_\epsilon(x) = \chi_{ \{ x:\ |m(x)-\mu|<\epsilon \} }(x)$ be the characteristic function of $\{ x: |m(x)-\mu|<\epsilon \}$. Choose a function $f_\epsilon\in L^2(\mathbb R^n)$ such that
$$ \|\chi_\epsilon f_\epsilon\|^2_{L^2(\RR)} = \int_{\mathbb R^n} |\chi_\epsilon f_\epsilon|^2dx   = \int_{ \{ x: |m(x)-\mu|<\epsilon \} } |f_\epsilon|^2dx=1. $$
We have
\begin{align*}
\| ( {M}_m-\mu ) (\chi_\epsilon f_\epsilon) \|^2_{L^2(\RR)} &= \int_{\mathbb R^n} |( { M}_m-\mu ) (\chi_\epsilon f_\epsilon)|^2dx \\
&= \int_{ \{ x\, :\, |m(x)-\mu|<\epsilon \} } |( { M}_m-\mu )|^2 |f_\epsilon|^2dx \\
&\leq \epsilon^2 \int_{ \{ x\,:\, |m(x)-\mu|<\epsilon \} }  |f_\epsilon|^2dx \\
&\leq \epsilon^2.
\end{align*}
This implies that $\mu\in\sigma_a({ M}_m)$, further $\mu\in \sigma_a(S_\varphi)$.

\medskip

Proof of (4): from (1) we see that $\sigma(S_\varphi) = \mathcal R(m)(\mathbb R^n)$. Hence, without loss of generality, we now just assume that $0\in  \mathcal R(m)(\mathbb R^n)$. Then for any $\epsilon >0$, we have
$$
\left\vert \{x: |m(x)|<\epsilon \}\right\vert>0.
$$
Choose a sequence of subsets in $\{x: |m(x)|<\epsilon \}$ such that
$$
E_{k+1}\subset E_k\subset \{x: |m(x)|<\epsilon \}$$
and
$
\left\vert E_k\right\vert\ne 0, $ $\left\vert E_k\right\vert\rightarrow 0 $  as $k\rightarrow \infty.$ Set
\begin{align}\label{f_k}
f_{k}(x)=\frac{1}{\sqrt{\left\vert E_k\right\vert}}\chi_{E_{k}}(x),
\end{align}
where $\chi_{E_k} $ is the characteristic function of $E_k$,
%then $\|f_{k}\|=1,$ and $f_{k}\rightarrow 0$ weakly in $L^2(\mathbb R^n).$
then $$ \|f_k\|^2_{L^2(\RR)} = \int_{E_k} {1\over { \left\vert E_k\right\vert}}dx=1 $$
and for any $g\in L^2(\mathbb R^n)$,
\begin{align*}
|\langle f_k, g\rangle_{L^2(\RR)}| &=  {1\over \sqrt{ \left\vert E_k\right\vert}} |\langle\chi_{E_k},g \rangle_{L^2(\RR)}|
\leq  {1\over \sqrt{ \left\vert E_k\right\vert}} \|\chi_{E_k}\|_{L^2(\RR)}\ \|\chi_{E_k}\, g\|_{L^2(\RR)}=\|\chi_{E_k}\, g\|_{L^2(\RR)}.
\end{align*}
Note that $g\in L^2(\mathbb R^n)$, we have that $\|\chi_{E_k}\, g\|\to0 $ as $k\to\infty$.
This implies that $f_k\to 0$ in $L^2(\mathbb R^n)$ in the weak sense.

It is not difficult to see that
\begin{align*}
\| ({ M}_m f_k) \|^2_{L^2(\RR)} &= \int_{E_{k}} |m f_k|^2dx +\int_{\mathbb R^n\setminus E_k}|m  f_k|^2dx 
= \int_{ E_k} |m f_k|^2dx 
\leq \epsilon^2 \int_{ E_k}  |f_k|^2dx 
=\epsilon^2.
\end{align*}
Since $\epsilon $ is arbitrary, we see that ${ M}_{m}$ is not Fredholm, that is $0\in \sigma_{e}({ M}_m),$ further $0\in \sigma_e(S_{\varphi }).$

\medskip

The proof of Theorem \ref{th spec} is complete.
\end{proof}

%\begin{cor}\label{cor4.3}
%Suppose  $\varphi\in \mathscr F^2(\mathbb C^n)$ such that $S_\varphi$ is bounded on $\mathscr F^2(\mathbb C^n)$ and $\varphi$ is defined as in \eqref{e1.2} for some $m\in L^\infty(\mathbb R^n).$
%Then  we have $\sigma(S_\varphi) = \mathcal R(m)(\mathbb R^n)$, where $\mathcal R(m)(\mathbb R^n)$
% is the essential range of $m$.
%\end{cor}
%\begin{proof}
%It is routine by the isomorphism $\mathfrak h:\ S_\varphi\to m$.
%\end{proof}
%\begin{rem}\label{rem5.5}
%$$
%\sigma(S_{\varphi })=\sigma_e(S_{\varphi }).
%$$

%\end{rem}

%\begin{thm}\label{th4.4}
%Suppose  $\varphi\in \mathscr F^2(\mathbb C^n)$  such that $S_\varphi$ is bounded on $\mathscr F^2(\mathbb C^n)$ and $\varphi$ is defined as in \eqref{e1.2} for some $m\in L^\infty(\mathbb R^n).$ Then $\sigma(S_\varphi)=\sigma_a(S_\varphi)$, where $\sigma_a(S_\varphi)$ denotes the approximate point spectrum of $S_\varphi$.
%\end{thm}
%\begin{proof}
%\end{proof}

%\begin{rem}
%Note that $S_\varphi = B\mathcal F^{-1 }T_m \mathcal F B^{-1}$. Hence we have that
%for $\mu\in \mathcal R(m)(\mathbb R^n)$,
%
%$$\mu\in \sigma_p(T_m) {\rm\ if\ and \ only\ if\ \ } L_n( E_\mu )>0,$$
%
%where
%$ E_\mu =\{ x:\ m(x)=\mu \}$.
%\end{rem}

%\begin{prop}\label{th4.6}
%Suppose $\varphi\in \mathscr F^2(\mathbb C^n)$  such that $S_\varphi$ is bounded on $\mathscr F^2(\mathbb C^n)$ and $\varphi$ is defined as in \eqref{e1.2} for some $m\in L^\infty(\mathbb R^n).$
%Then $\mu\in \mathcal R(m)(\mathbb R^n)$ is the eigenvalue of $S_\varphi$ if and only if
%$$ \left\vert \{x: m(x)=\mu\} \right\vert>0.  $$
%\end{prop}
%\begin{proof}
%\end{proof}

\medskip
\noindent {\bf 5.2.3.\ Compactness of the Operator $S_\varphi$}.\ Next we provide the proof of the compactness of the operator $S_\varphi$.

\begin{thm}\label{th4.7}
Suppose $\varphi\in \mathscr F^2(\mathbb C^n)$ such that $S_\varphi$ is bounded on $\mathscr F^2(\mathbb C^n)$ and $\varphi$ is defined as in \eqref{e1.2} for some $m\in L^\infty(\mathbb R^n).$ Then $S_\varphi$
is compact if and only if $\varphi=0$.
\end{thm}
\begin{proof}
We need only to prove that $S_\varphi$ can not be compact if $\varphi\not=0$. Since $\varphi\not=0$, we see that $m\not=0$.

Write $E_0=\{x: m(x)\not=0\}$. Then $\left\vert E_0\right\vert>0$. Thus, there is an $\epsilon_0>0$ such that
$E_{\epsilon_0} =\{ x: |m(x)|>\epsilon_0\}$ has positive measure. Without loss of generality, assume that
$0<\left\vert E_{\epsilon_0}\right\vert <\infty$. Choose a sequence of subsets in $E_{\epsilon_0}$ such that
$E_{\epsilon_0} \supset E_k \supset E_{k+1}$, and $\left\vert E_k\right\vert>0$, $\lim\limits_{k\to \infty} \left\vert E_k\right\vert =0.$
Let $f_k(x)$ be defined as in \eqref{f_k}. Then from the argument as in the proof of (4) of Theorem \ref{th spec}, we see that $f_k\to 0$ in $L^2(\mathbb R^n)$ in the weak sense.
%$$ f_k(x) = {1\over \sqrt{ L_n(E_k)}} \chi_{E_k}(x), $$

It is obvious that
\begin{align*}
\|{ M}_mf_k\|^2_{L^2(\RR)} = \int_{\mathbb R^n} |m\cdot f_k|^2dx \geq \int_{ E_k} |m\cdot f_k|^2dx
\geq \epsilon_0^2 \int_{ E_k} | f_k|^2dx =\epsilon_0^2\not\to0.
\end{align*}
This shows that ${ M}_m$ is not compact, and hence $S_\varphi$ can not be compact.
\end{proof}

%\bigskip

%\bigskip

\subsection{Invariant  subspaces of $S_{\varphi}$}
%{(For intrduction)
The well-known Beurling theorem characterizes the invariant subspace lattice of the coordinate multiplier ${ M}_z$ on
 the Hardy space $H^2(\mathbb{T})$ of the unit circle  $\mathbb{T}$ (see  \cite{Do1, Ga}). However, it is very difficult
  to obtain the characterization of the invariant subspace lattice of a general bounded linear operator $T$ even if $T $ is normal.
%\medskip
One possible attempt arises from observing that the reducing subspaces of a normal operator may be determined by 
its spectral projections. However, in general, one does not know the explicit form of the spectral projections.

In this subsection, we  characterize the reducing subspaces of ${ M}_m$ for any $m\in L^{\infty }(\mathbb{R}^n).$ 
Moreover, based on our main result Theorem \ref{th1.1}, we can further obtain the characterization of the reducing 
subspaces of $S_{\varphi }$ with $\varphi$  defined as \eqref{e1.2} for some $m\in L^{\infty }(\mathbb{R}^n)$.

It is easy to prove that for $m\in L^{\infty }(\mathbb{R}^n)$, $R({ M}_m)$ is closed if and only if either $0\notin \mathscr{R}(m), $
or $0\in \mathscr{R}(m),$ but $m$ is essentially lower bounded on $\operatorname{supp} m,$ the support of $m.$  In particular,
 if $E \subset \mathbb{R}^n$  with $\left\vert E\right\vert>0,$ then for $\chi_{E},$ the characteristic function of $E$,
  $X_0=\chi_{E} L^2(\mathbb{R}^n)$ is an invariant subspace or zero subspace of ${ M}_{m}.$ Thus we have the following.

\begin{thm} Suppose $m\in L^{\infty }(\mathbb{R}^n)$ and  $ \varphi\in \mathscr{F}^2(\mathbb{C}^n)$ is defined as in \eqref{e1.2}.
  Let $X$ be a subspace of $\mathscr{F}^2(\mathbb{C}^n).$ Then ${  X}$ is a reducing subspace of $S_{\varphi }$  if and only if there is a set $E\subset \mathscr{R}(m)$ with $\left\vert E\right\vert>0,$  such that
$$
X=S_{\varphi_0}\mathscr{F}^2(\mathbb{C}^n),
$$
where $\varphi_0=\int_{\mathbb{R}^n}\chi_{E}(x)e^{-2\left(x-\frac{i}{2}   z  \right)^2} dx.$
\end{thm}
\begin{proof}
Let $P$ be the orthogonal projection from $ \mathscr{F}^2(\mathbb{C}^n)$ to $X.$ If $X$ is a reducing  subspace 
of $S_{\varphi },$ then $PS_{\varphi}=S_{\varphi}P.$  Clearly, $P$ is   the spectral projection of $S_{\varphi }$.
%Note the spectral measure of $S_{\varphi}$ is unique,
Thus $PS_\psi =S_\psi P$ for any $S_\psi\in  {\mathscr A}$ since  ${\mathscr A}$  is maximal commutative.
We see that there is a $E\subset \mathscr{R}(m)$ with $\left\vert E\right\vert>0,$  such that $P=S_{\varphi_0},$ 
where  $\varphi_0=\int_{\mathbb{R}^n}\chi_{E}(x)e^{-2\left(x-\frac{i}{2}   z  \right)^2} dx.$ Thus
$$
X=P\mathscr{F}^2(\mathbb{C}^n)=S_{\varphi_0}\mathscr{F}^2(\mathbb{C}^n).
$$

Conversely, if  there is a $\varphi_0\in  \mathscr{F}^2(\mathbb{C}^n)$ with $\chi_{E}, E\subset \mathscr{R}(m) $ such that $X=S_{\varphi_0} \mathscr{F}^2(\mathbb{C}^n),$  then $X$ is a closed subspace. By noting that  $S_{\varphi}S_{\varphi_0}=S_{\varphi_0}S_{\varphi},$ and $S_{\varphi_0}^2=S_{\varphi_0},  $ $S_{\varphi_0}^{*}=S_{\varphi_0},$ we see that $S_{\varphi_0}$ is a projector which commutes with $S_{\varphi}.$ Hence, $X=S_{\varphi_0} \mathscr{F}^2(\mathbb{R}^n)$ is the reducing subspace of $S_{\varphi}.$
\end{proof}

%In general, it is impossible to get a complete representation of the invariant subspace lattice of a bounded linear operator  $T$  even if $T$ is normal.
We now recall \cite[Theorem 6.9]{Do1} which gives the characterization of the simple invariant subspaces of the coordinate multiplier on  $L^2(\mathbb{T})$.
\begin{lem}[{\cite[Theorem 6.9]{Do1}}]If $\mu $ is a positive regular Borel measure on $\mathbb{T}$, 
then a non-trivial closed subspace $X$ of $L^2(\mu )$  satisfies ${ M}_{z}X\subset X$ and
 $\cap_{n\geq 0}{ M}_{z^n}X=\{0\}$ if and only if there exists a Borel function $m$ on $\mathbb{T}$ such that
  $|m|^2d\mu =d\theta /2\pi $ and $X=m H^2(\mathbb{T}).$
\end{lem}
By the connection between the Hardy space $H^2(\mathbb{T})$ and $H^2(i\mathbb{R}),$ we may characterize the simple invariant 
subspaces of ${ M}_{\varphi },$ where $\varphi(w)=\frac{w-1}{w+1}.$

We say that $X$ is the simple invariant subspace of ${ M}_z$.
\begin{thm}\label{th5.16}
Suppose $\varphi =\frac{w-1}{w+1}$ is the Riemann map  from $\mathbb{C}_+$ to $\mathbb{D},$
  ${ M}_{\varphi }$ is the multiplier on $L^2(i\mathbb{R})$ defined as ${ M}_{\varphi }f=\varphi f$
   for any $f\in L^2(i\mathbb{R}).$  Then a non-trivial closed subspace $X$ of $L^2(i\mathbb{R})$ 
   satisfies 
   $
   { M}_{\varphi }X\subset X $ and $  \cap_{n\geq 0}{ M}_{\varphi^n}X=\{0\}
   $ if and only if 
   there is a Borel function $m$ on $\mathbb{T}$ such that $|m|=1 a.e.$ and
$$
X =\big\{ (m\circ \varphi ) \varphi_{0}H^2(i\mathbb{R}) \big\} \ \ \ {\rm with}\ \ \varphi_0=\frac{1+it}{\sqrt{1+t^2}}.
$$
%where $\varphi_0=\frac{1+it}{\sqrt{1+t^2}}.$
Moreover,
$BF^{-1}X$ is the simple invariant subspace  of $S_{\psi},$
where   $
\psi =\int_{\mathbb{R}}\varphi e^{-2\left(x-\frac{i}{2}   z  \right)^2} dx.
$
\end{thm}

\begin{proof}
We need only to prove that $X$, a simple invariant subspace of ${ M}_{\varphi }$, 
must have the form $(m\circ \varphi )\varphi_{0}H^2(i \mathbb{R})$ for some $m\in L^{\infty }(\mathbb{T})$ with $|m|=1$\ a.e..  Write
$$
\tilde{X}=C_{\varphi^{-1}}\Big(\frac{1}{1-\varphi}X\Big),
$$
where $C_{\varphi^{-1}}f=f\circ\varphi^{-1}$.  Then for any $\tilde{f}\in \tilde{X}, $ there is an $f\in X$ such that
$$
\tilde{f}=C_{\varphi^{-1}}\Big(\frac{1}{1-\varphi}f\Big)=\frac{1}{1-z}C_{\varphi^{-1}}f\in L^2(\mathbb{T}) .
$$
In fact,  for any measurable function $g$ on $\mathbb{T}$, we have
$$
\int_{\mathbb T}g(e^{i\theta}){d\theta \over 2\pi}=\int_{\mathbb{R}}g\circ \varphi (it)\frac{1}{1+t^2} {dt\over \pi}
$$
(see \cite{Ho1}), thus
\begin{align*}
\|\tilde f\|^2_{L^2(\mathbb{T})}&=\int_{\mathbb T}|\tilde{f}(e^{i\theta})|^2  {d\theta \over 2\pi}
 =\int_{\mathbb{R}}|\tilde{f}\circ \varphi (it)|^2\frac{1}{1+t^2}{dt\over \pi}\\
&=\int_{\mathbb{R}}\Bigg|\bigg[C_{\varphi^{-1}}\Big(\frac{1}{1-\varphi}f\Big)\bigg]\circ \varphi \Bigg|^2\frac{1}{1+t^2}{dt\over \pi}\\
&=\int_{\mathbb R}\Big|\frac{1}{1-\varphi}f\Big|^2\frac{1}{1+t^2}{dt\over \pi}\\
&=\frac{1}{4}\int_{\mathbb R}|f(it)|^2{dt\over \pi}\\
&=\frac{1}{4\pi}\|f\|^2_{L^2(i\mathbb{R})},
\end{align*}
that is,   $\|\tilde{f}\|_{L^2(\mathbb{T})}=\frac{1}{2\sqrt\pi}\|f\|_{L^2(i\mathbb{R})}$. Hence $\tilde{X}$ is closed. For arbitrary $g\in \tilde{X},$
 there is an $f\in X$ such that $g=C_{\varphi}^{-1}\big(\frac{1}{1-\varphi}f\big).$ Then
\begin{align*}
{ M}_zg&=zC_{\varphi}^{-1}\Big(\frac{1}{1-\varphi}f\Big)
=\Big(C_{\varphi}^{-1}\varphi\Big)C_{\varphi}^{-1}\Big(\frac{1}{1-\varphi}f\Big)=C_{\varphi}^{-1}\Big[\frac{1}{1-\varphi}(\varphi f )\Big].
\end{align*}
It is routine to check that ${ M}_{z^n}g=C_{\varphi }^{-1}[\frac{1}{1-\varphi}(\varphi^nf)]. $
Since $\varphi^n f\in X,$ we see that $\tilde{X}$ is a simple invariant subspace of ${ M}_z$.
Thus there is a $m\in L^{\infty }(\mathbb{T})$ with $|m|=1$ a.e. such that
$$
\tilde{X}=mH^2(\mathbb{T}).$$
On the other hand,
$$
C_{\varphi }\tilde{X}=C_{\varphi}C_{\varphi^{-1}}\Big(\frac{1}{1-\varphi }X\Big),
$$
we see that
$$
\frac{1}{1-\varphi}X=(m\circ\varphi )(1+it)H^2(i\mathbb{R}),
$$
since $C_{\varphi }H^2(\mathbb{T})=(1+it)H^2(i\mathbb{R}).$ Further,
$$
X=(m\circ\varphi )(1-\varphi)(1+it)H^2(i\mathbb{R}).
$$
Note
$
|1-\varphi|=\frac{2}{\sqrt{1+t^2}},$
write $\varphi_{0}=\frac{1+it}{\sqrt{1+t^2}},$ then $|\varphi_0|=1,$ and
$$
X=(m\circ\varphi )\varphi_0H^2(i\mathbb{R}).
$$
By the Fourier transform and Bargmann transform, we know that $BF^{-1}X$ is a simple 
invariant subspace of $S_{\psi},$ completing the proof of Theorem \ref{th5.16}.
\end{proof}

\section{Concluding Remarks}

This paper has studied a new class of operators on the Fock space which has totally different properties from the well-known Toeplitz operator.  For example, for any $\varphi\in L^2(\mathbb C^n)$, if the Toeplitz operator $T_\varphi$ is bounded, then  $T_\varphi^* = T_{\bar \varphi}$; for any analytic function $\varphi$, $T_\varphi$ is subnormal, and moreover, $T_\varphi$ is normal if and only if $\varphi$ is constant.  While $S_\varphi$ has better properties as evidenced by the theorems above. Moreover, $S_\varphi$ connects singular integrals in harmonic analysis to operators in the complex setting via the Bargmann transform.  This connection enables the resolution of problems in complex analysis via techniques from harmonic analysis.

%{\color{red}Closely related questions are: What is the form of the Toeplitz operator after application of the Bargmann transform?  Can we also apply harmonic analysis techniques to study properties of Toeplitz operators? These questions will guide the next steps in this area.}

\bigskip
\noindent
{\bf Acknowledgements}:
We would like to thank the referee for a careful reading of this manuscript and providing many helpful comments and suggestions.

G. Cao is supported by the NNSF
of China, Grant No. ~11671152.
J. Li is supported by the Australian Research Council (ARC) through the
research grant DP170101060 and by Macquarie University Research Seeding Grant.
B. D. Wick is supported in part by National Science Foundation DMS grant $\#$1500509 and $\#$1800057.
 L. Yan is supported by the NNSF
of China, Grant No. ~11521101 and  ~11871480,   and   by the Australian Research Council (ARC) through the
research grant  DP190100970.
G. Cao and L. Yan   thank K. Zhu for helpful discussions.

\end{document}